\documentclass[a4paper,twoside,10pt]{article}

\usepackage[english]{babel}
\usepackage[utf8]{inputenc}
\usepackage{lmodern,color}
\usepackage[T1]{fontenc}
\usepackage{indentfirst}
\usepackage{amsmath,amssymb}
\usepackage[percent]{overpic}
\usepackage[a4paper,top=3cm,bottom=3cm,left=3cm,right=3cm,%
bindingoffset=5mm]{geometry}
\usepackage{lmodern}
\usepackage[T1]{fontenc}
\usepackage{lettrine}
\usepackage{booktabs}

\usepackage{authblk}
\usepackage{emp}
\usepackage{amsthm}
\usepackage{amsmath,amssymb,amsthm}
\usepackage{braket}
\usepackage{chngpage}
\usepackage{cases}
\usepackage{graphicx}
\usepackage{epstopdf}
\usepackage{tabularx}
\usepackage{multirow}
\usepackage{enumerate}
\usepackage{todonotes}

\newcommand{\numberset}{\mathbb}
\newcommand{\N}{\numberset{N}}
\newcommand{\R}{\numberset{R}}

\newcommand{\B}{\numberset{B}}
\newcommand{\Pk}{\numberset{P}}
\newcommand{\OO}{\mathcal{O}}
\renewcommand{\epsilon}{\varepsilon}
\renewcommand{\theta}{\vartheta}
\renewcommand{\rho}{\varrho}
\renewcommand{\phi}{\varphi}
\newcommand{\pp}{\boldsymbol{p}}
\newcommand{\xx}{\boldsymbol{x}}
\newcommand{\uu}{\boldsymbol{u}}
\newcommand{\vv}{\boldsymbol{v}}
\newcommand{\zz}{\boldsymbol{z}}
\newcommand{\ww}{\boldsymbol{w}}
\newcommand{\ff}{\boldsymbol{f}}
\newcommand{\nn}{\boldsymbol{n}}
\newcommand{\tf}{\boldsymbol{t}}

\newcommand{\ffi}{\boldsymbol{\phi}}
\newcommand{\ppsi}{\boldsymbol{\psi}}
\newcommand{\ssigma}{\boldsymbol{\sigma}}
\newcommand{\Sg}{\boldsymbol{\Sigma}}

\newcommand{\dd}{{\rm div}}
\newcommand{\gr}{\nabla}
\newcommand{\Gr}{\boldsymbol{\nabla}}
\newcommand{\rr}{{\rm rot}}
\newcommand{\RR}{\boldsymbol{{\rm rot}}}

\newcommand{\CC}{\boldsymbol{{\rm curl}}}
\newcommand{\dl}{\boldsymbol{\Delta}}
\newcommand{\VV}{\boldsymbol{V}}

\newcommand{\ZZ}{\boldsymbol{Z}}

\def\P0{{\Pi^{0, E}_k}}

\def\PP0{{\boldsymbol{\Pi}^{0, E}_{k-1}}}

\newcommand{\epseps}{\boldsymbol{\epsilon}}

{\left\lbrace\begin{array}{@{}l@{}}}%
{\end{array}\right.}
\theoremstyle{definition}

\theoremstyle{remark}
\newtheorem{remark}{Remark}[section]

\theoremstyle{remark}

\theoremstyle{plain}
\newtheorem{theorem}{Theorem}[section]
\newtheorem{proposition}{Proposition}[section]

\newtheorem{lemma}{Lemma}[section]

\usepackage{lipsum}

\usepackage{authblk}
\usepackage{color}

\usepackage{setspace}

\author[1]{L. Beir\~ao da Veiga \thanks{lourenco.beirao@unimib.it}}

\author[1]{F. Dassi \thanks{franco.dassi@unimib.it}}

\author[1]{G. Vacca \thanks{giuseppe.vacca@unimib.it}}

\affil[1]{Dipartimento di Matematica e Applicazioni,
Universit\`a degli Studi di Milano Bicocca,
Via Roberto Cozzi 55 - 20125 Milano, Italy}

\title{ \textbf{The Stokes complex for Virtual
Elements in three dimensions}}

\date{\today}

\begin{document}

\maketitle

\begin{abstract}
The present paper has two objectives. On one side, we develop and test numerically divergence free Virtual Elements in three dimensions, for variable ``polynomial'' order. These are the natural extension of the two-dimensional divergence free VEM elements, with some modification that allows for a better computational efficiency. We test the element's performance both for the Stokes and (diffusion dominated) Navier-Stokes equation. 
The second, and perhaps main, motivation is to show that our scheme, also in three dimensions, enjoys an underlying discrete Stokes complex structure.
We build a pair of virtual discrete spaces based on general polytopal partitions, the first one being scalar and the second one being vector valued, such that when coupled with our velocity and pressure spaces, yield a discrete Stokes complex.
\end{abstract}

\section{Introduction}

\bigskip

The Virtual Element Method (VEM) was introduced in \cite{volley,autostoppisti} as a generalization of the Finite Element Method (FEM) allowing for general polytopal meshes. Nowadays the VEM technology has reached a good level of success; among the many papers we here limit ourselves in citing a few sample works 
\cite{
GPDM:2015,
BPP:2017,
BB:2017,
dassi-mascotto:2018,
BBDMRbis:2018,
MPP:2018,
AMV:2018,
AHHW:2018}. 
It was soon recognized that the flexibility of VEM allows to build elements that hold peculiar advantages also on more standard grids. One main example is that of ``divergence-free'' Virtual Elements for Stokes-type problems, initiated in \cite{BLV:2017, ABMV:2014} and further developed in \cite{BLV:2018,vacca:2018}. 
An advantage of the proposed family of Virtual Elements is that, without the need of a high minimal polynomial degree as it happens in conforming FEM, it is able to yield a discrete divergence-free (conforming) velocity solution, which can be an interesting asset as explored for Finite Elements in \cite{guzman-neilan:2014,john-linke-merdon-neilan-rebholz:2017,neilan-sap:2016,linke-merdon:2016,guzaman-neilan:2018}. 
For a wider look in the literature, other VEM for Stokes-type problems can be found in \cite{LLC:2017,CGM:2016,CGS:2018,GMS:2018,CW:2018,dassi-vacca:2018} while different polygonal methods for the same problem in \cite{LVY:2014, CFQ:2017, BdPD:2018, dPK:2018}.

The present paper has two objectives. On one side, we develop and test numerically for the first time the divergence free Virtual Elements in three dimensions (for variable ``polynomial'' order $k$). These are the natural extension of the two-dimensional VEM elements of \cite{BLV:2017,BLV:2018}, with some modification that allows for a better computational efficiency. We first test the element's performance for the Stokes and (diffusion dominated) Navier-Stokes equation for different kind of meshes (such as Voronoi, but also cubes and tetrahedra) and then  show a specific test that underlines the divergence free property (in the spirit of \cite{linke-merdon:2016, BLV:2018}). 

The second, and perhaps main, motivation is to show that our scheme, also in three dimensions, enjoys an underlying discrete Stokes complex structure. That is, a discrete structure of the kind
$$
\R \, \xrightarrow[]{ \,\quad \text{{$i$}} \quad \,   } \,
W_h \, \xrightarrow[]{   \quad \text{{$\gr$}} \quad  }\,
\Sg_h \, \xrightarrow[]{ \, \, \, \text{{$\CC$}} \, \, \, }\,
\VV_h \, \xrightarrow[]{  \, \, \, \, \text{{$\dd$}} \, \, \, \, }\,
Q_h \, \xrightarrow[]{ \quad 0 \quad }
0 
$$
where the image of each operator exactly corresponds to the kernel of the following one, thus mimicking the continuous complex
$$
\R \, \xrightarrow[]{ \,\quad \text{{$i$}} \quad \,   } \,
H^1(\Omega) \, \xrightarrow[]{   \quad \text{{$\gr$}} \quad  }\,
\boldsymbol{\Sigma}(\Omega) \, \xrightarrow[]{ \, \, \, \text{{$\CC$}} \, \, \, }\,
[H^1(\Omega)]^3 \, \xrightarrow[]{  \, \, \, \, \text{{$\dd$}} \, \, \, \, }\,
L^2(\Omega) \, \xrightarrow[]{ \quad 0 \quad }
0 \,,
$$
with $\boldsymbol{\Sigma}(\Omega)$ denoting functions of $L^2(\Omega)$ with ${\bf curl}$ in $H^1(\Omega)$ 
\cite{girault-raviart:book,amrouche-et-al:1998}.
Discrete Stokes complexes has been extensively studied in the literature of Finite Elements since the presence of an underlying complex implies a series of interesting advantages (such as the divergence free property), in addition to guaranteeing that the discrete scheme is able to correcly mimic the structure of the problem under study \cite{demkowicz-et-al:2000,demkowicz-buffa:2005,arnold-falk-winter:2006,arnold-falk-winther:2006bis,arnold-falk-winter:2010,buffa-rivas-sangalli-velasquez:2011,falk-neilan:2013,evans-hughes:2013,neilan:2015}. This motivation is therefore mainly theoretical in nature, but it serves the important purpose of giving a deeper foundation to our method. We therefore build a pair of virtual discrete spaces based on general polytopal partitions of $\Omega$, the first one $W_h$ which is conforming in $H^1(\Omega)$ and the second one $\Sg_h$ conforming in $\boldsymbol{\Sigma}(\Omega)$, such that, when coupled with our velocity and pressure spaces, yield a discrete Stokes complex. We also build a set of carefully chosen associated degrees of freedom. This construction was already developed in two dimensions in \cite{BMV:2018}, but here things are more involved due to the 
much more complex nature of the curl operator in 3D when compared to 2D. 
In this respect we must underline that, to the best of the authors knowledge, no Stokes exact complex of the type above exists for conforming Finite Elements in three dimensions. There exist FEM for different (more regular) Stokes complexes, but at the price of developing cumbersome elements with a  large minimal polynomial degree (we refer to \cite{john-linke-merdon-neilan-rebholz:2017} for an overview) or using a subdivision of the element \cite{christiansen:2018}. 
We finally note that our construction holds for a general ``polynomial'' order $k \ge 2$. 

The paper is organized as follows. After introducing some notation and preliminaries in Section \ref{sec:notations}, the Virtual Element spaces and the associated degrees of freedom are deployed in Section \ref{sec:spaces}. In Section \ref{sec:derham} we prove that the introduced spaces constitute an exact complex. In Section \ref{sec:VEM-ns} we describe the discrete problem, together with the associated projectors and bilinear forms. In Section \ref{sec:num_test} we provide the numerical tests. Finally, in the appendix we prove a useful lemma.

\section{Notations and preliminaries}
\label{sec:notations}

In the present section we introduce some basic tools and notations useful in the construction and theoretical analysis of Virtual Element Methods.

Throughout the paper, we will follow the usual notation for Sobolev spaces
and norms \cite{Adams:1975}.
Hence, for an open bounded domain $\omega$,
the norms in the spaces $W^s_p(\omega)$ and $L^p(\omega)$ are denoted by
$\|{\cdot}\|_{W^s_p(\omega)}$ and $\|{\cdot}\|_{L^p(\omega)}$ respectively.
Norm and seminorm in $H^{s}(\omega)$ are denoted respectively by
$\|{\cdot}\|_{s,\omega}$ and $|{\cdot}|_{s,\omega}$,
while $(\cdot,\cdot)_{\omega}$ and $\|\cdot\|_{\omega}$ denote the $L^2$-inner product and the $L^2$-norm (the subscript $\omega$ may be omitted when $\omega$ is the whole computational
domain $\Omega$).

\subsection{Basic notations and mesh assumptions}
\label{sub:mesh}
From now on, we will denote with $P$ a general  polyhedron 
having $\ell_V$ vertexes $V$, $\ell_e$ edges $e$ and $\ell_f$ faces $f$.
\\
For each polyhedron $P$, each face $f$ of $P$ and each edge $e$ of $f$ we denote with:
\begin{itemize}
\item[-] $\nn_P^f$ (resp. $\nn_P$) the unit outward normal vector to  $f$ (resp. to $\partial P$),
\item[-] $\nn_f^e$ (resp. $\nn_f$) the unit vector in the plane of $f$ that is normal to the edge $e$ (resp. to $\partial f$) and outward with respect to $f$,
\item[-] $\tf_f^e$ (resp. $\tf_f$) the unit vector in the plane of $f$ tangent to $e$ (resp.  to $\partial f$) counterclockwise with respect to $\nn_P^f$, 
\item[-] $\boldsymbol{\tau}_1^f$ and $\boldsymbol{\tau}_2^f$ two orthogonal unit vectors lying on $f$ and such that $\boldsymbol{\tau}_1^f \wedge \boldsymbol{\tau}_2^f = \nn_P^f$,
\item[-] $\tf_e$ a unit vector tangent to the edge $e$.  
\end{itemize}
Notice that the vectors $\tf_f^e$, $\tf_f$, $\boldsymbol{\tau}_1^f$ and $\boldsymbol{\tau}_2^f$ actually depend on $P$ (we do not write such dependence explicitly for lightening the notations).

In the following $\OO$ will denote a general geometrical  entity (element, face, edge) having diameter $h_{\OO}$.

Let  $\Omega$ be the computational domain that we assume to be a contractible polyhedron  (i.e. simply connected polyhedron with  boundary $\partial \Omega$ which consists of one connected component),  
with Lipschitz boundary.
Let $\set{\Omega_h}_h$ be a sequence of decompositions of $\Omega$ into general polyhedral elements $P$ 
where
$
h := \sup_{P \in \Omega_h} h_P
$.

We suppose that for all $h$, each element $P$ in $\Omega_h$ is a contractible polyhedron that fulfils the following assumptions:
\begin{description}
\item [$\mathbf{(A1)}$] $P$ is star-shaped with respect to a ball $B_P$ of radius $ \geq\, \rho \, h_P$, 
\item [$\mathbf{(A2)}$] every face $f$ of $P$ is star-shaped with respect to a disk $B_f$ of radius $ \geq\, \rho \, h_P$,
\item [$\mathbf{(A3)}$] every edge $e$ in $P$ satisfies 
$ h_e \geq \rho \, h_P$, 
\end{description}
where $\rho$ is a uniform positive constant. We remark that the hypotheses $\mathbf{(A1)}$, $\mathbf{(A2)}$ and $\mathbf{(A3)}$, though not too restrictive in many practical cases, 
can be further relaxed, as investigated in ~\cite{BLR:2017, brenner-guan-sung:2017, brenner-sung:2018, cao-chen:2018}. 

The total number of vertexes, edges, faces and elements in the decomposition $\Omega_h$ are denoted respectively with $L_V$, $L_e$, $L_f$, $L_P$. 

%
For any mesh object $\OO$ and for $n \in \N$  we introduce the spaces:
\begin{itemize}
\item $\Pk_n(\OO)$ the  polynomials on $\OO$ of degree $\leq n$  (with the extended notation $\Pk_{-1}(\OO)=\{0\}$),

\item $\widehat{\Pk}_{n \setminus m}(\OO) := 
\Pk_{n}(\OO) \setminus \Pk_{m}(\OO)$
for $m \leq n$, denotes the polynomials in $\Pk_{n}(\OO)$ with monomials of degree strictly greater than $m$.
\end{itemize}
Moreover for any mesh object $\OO$ of dimension $d$ we define
\begin{equation}
\label{eq:poly_dim}
\pi_{n, d} := \dim(\Pk_{n}(\OO)) = \dim(\Pk_{n}(\R^d)) \,,
\end{equation}
and thus $\dim(\widehat{\Pk}_{n \setminus m}(\OO)) = \pi_{n, d} -\pi_{m, d}$.

In the following the symbol $\lesssim$ will denote a bound up to a generic positive constant,
independent of the mesh size $h$,  but which may depend on $\Omega$,
on the ``polynomial'' order $k$  and
on the shape constant $\rho$ 
in assumptions $\mathbf{(A1)}$, $\mathbf{(A2)}$ and $\mathbf{(A3)}$.

%
%
%
%
%
%

\subsection{Vector calculus \& de Rham complexes}
\label{sub:vector_calculus}

Here below we fix some additional notation of the multivariable calculus.
\\
\textbf{Three dimensional operators.}
In three dimensions we denote with $\xx = (x_1, \, x_2, \, x_3)$ the independent variable. 
With a usual notation   
the symbols $\gr$ and $\Delta$ 
denote the gradient and Laplacian 
for scalar functions, while 
$\dl $, $\Gr$, $\epseps$,  $\dd$ and $\CC$ denote the vector Laplacian,  the gradient and the symmetric gradient operator, the divergence and the curl operator for vector fields.
Note that on each polyhedron $P$ the following useful polynomial decompositions hold
\begin{align}
\label{eq:poly_decomposition_3D_grad}
[\Pk_n(P)]^3 &= \nabla(\Pk_{n+1}(P))  \oplus (\xx \wedge [\Pk_{n-1}(P)]^3)\,,
\\
\label{eq:poly_decomposition_3D_curl}
[\Pk_n(P)]^3 &= \CC(\Pk_{n+1}(P))  \oplus \xx \,\Pk_{n-1}(P) \,.
\end{align}
\textbf{Tangential operators.} Let $f$ be a face of a polyhedron $P$, we denote with $\xx_f:= ({x_{f}}_1, \, {x_{f}}_2)$ the independent variable (i.e. a local coordinate system on $f$ associated with the axes
$\boldsymbol{\tau}_f^1$ and $\boldsymbol{\tau}_f^2$).
The tangential differential operators are denoted by a subscript $f$. Therefore 
the symbols $\gr_f$ and $\Delta_f$  denote the gradient and Laplacian  for scalar functions, while
$\dl_f$, $\Gr_f$, and $\dd_f$ denote the vector Laplacian,  the gradient operator
and the divergence  for vector fields
on $f$ (with respect to the coordinate $\xx_f$).
Furthermore for a scalar function $\phi$ and a vector field $\vv := (v_1, \, v_2)$ we set
\[
\RR_f \, \phi := \left( \frac{\partial \phi}{\partial {x_{f}}_2}, \, -\frac{\partial \phi}{\partial {x_{f}}_1} \right) 
\qquad \text{and} \qquad
\rr_f \, \vv := \frac{\partial v_2}{\partial {x_{f}}_1}  -  \frac{\partial v_1}{\partial {x_{f}}_2} \,.
\]
The following 2-d polynomial decompositions hold
\begin{align*}
[\Pk_n(f)]^2 &= \nabla_f(\Pk_{n+1}(f))  \oplus \xx_f^{\perp} \,\Pk_{n-1}(f) \,,
\\
[\Pk_n(f)]^2 &= \RR_f(\Pk_{n+1}(f))  \oplus \xx_f \,\Pk_{n-1}(f) \,,
\end{align*}
 where $\xx_f^{\perp}:= ({x_{f}}_2, \, -{x_{f}}_1)$.
 
Given a 3-d vector valued function $\vv$ defined in $P$, the tangential component $\vv_{f}$
of $\vv$ with respect to the face $f$ is defined by
\[
\vv_{f} := \vv - (\vv \cdot \nn_P^f)\nn_P^f \,.
\] 
Noticing that $\vv_{f}$ is a 3-d vector field tangent to $f$, with a slight abuse of notations we define the 2-d vector field $\vv_{\tau}$ on $\partial P$, such that on each face $f$ its restriction to the face $f$ satisfies
\[
\vv_{\tau}(\xx_f) := \vv_{f}(\xx) \,.
\]
The 3-d function $\vv$ and its 2-d tangential restriction $\vv_{\tau}$ are related by the  $\CC$-$\rr$ compatibility condition 
\begin{equation}
\label{eq:compatibility_CC_rr}
\CC \, \vv \cdot \nn_P^f = \rr_f \, \vv_{\tau} \qquad \text{on any $f \in \partial P$.}
\end{equation}
Moreover the Gauss theorem ensures the following $\rr$-tangent component relation
\begin{equation}
\label{eq:compatibility_rr_tf}
\int_f \rr_f \, \vv_{\tau} \,{\rm d}f = \int_{\partial f} \vv \cdot \tf_f \, {\rm d}s \qquad \text{for any $f \in \partial P$.}
\end{equation}
Finally, for any scalar function $v$ defined in $P$, we denote with $v_{\tau}$ the scalar function defined in $\partial P$  such that
\[
v_{\tau}(\xx_f) := v(\xx)_{|f} \qquad \text{on each face $f \in \partial P$.}
\]
\vspace{2ex}

On a generic mesh object $\OO$ with geometrical dimension $d$, on a face $f$ and on a polyhedron $P$ we define
following the functional spaces: 
\begin{align*}
 L^2_0(\OO) &:= \{ v \in L^2(\OO) \quad \text{s.t.} \quad \int_{\OO} v \, {\rm d}\OO = 0  \}\\
\ZZ(\OO) &: \{ \vv \in [H^1(\OO)]^2 \quad \text{s.t.} \quad \dd \,\vv = 0 \quad \text{in $\OO$} \} \\
{\boldsymbol H}(\dd, \, \OO)  &:= \{ \vv \in [L^2(\OO)]^d \quad \text{with} \quad \dd \,\vv  \in L^2(\OO) \}
\\
{\boldsymbol H}(\rr, \, f)  &:= \{ \vv \in [L^2(f)]^2 \quad \text{with} \quad \rr_f \,\vv  \in L^2(f) \}
\\
{\boldsymbol H}(\CC, \, P)  &:= \{ \vv \in [L^2(P)]^3 \quad \text{with} \quad \CC \,\vv  \in [L^2(P)]^3 \}
\\
\boldsymbol{\Sigma}(P)  &:= \{ \vv \in [L^2(P)]^3 \quad \text{with} \quad \CC \,\vv  \in [H^1(P)]^3 \}
\\
\boldsymbol{\Psi}(P) &:= \{\vv \in H(\dd, \, P) \cap H(\CC, \, P)
\quad \text{s.t.} \quad \dd \, \vv \in H^1(P) \,,  \quad 
 \CC \, \vv \in  [H^1(P)]^3
 \} 
\end{align*}
with the ``homogeneous counterparts''
\begin{align*}
\ZZ_0(\OO) &:= \{ \vv \in \ZZ(\OO) \quad \text{s.t.} \quad \vv = {\boldsymbol 0} \quad \text{on $\partial \OO$} \} \\
{\boldsymbol H}_0(\dd, \, \OO)  &:= \{ \vv \in H(\dd, \, \OO) \quad \text{s.t.} \quad  \vv \cdot \nn_{\OO} = 0 \quad \text{on $\partial \OO$} \}
\\
{\boldsymbol H}_0(\rr, \, f)  &:= \{ \vv \in H(\rr, \, f) \quad \text{s.t.} \quad  \vv \cdot \tf_{f} = 0 \quad \text{on $\partial f$} \}
\\
{\boldsymbol H}_0(\CC, \, P)  &:= \{ \vv \in H(\CC, \, P) \quad \text{s.t.} \quad  \vv_{\tau} = {\boldsymbol 0} \quad \text{on $\partial P$} \}
\\
 \boldsymbol{\Sigma}_0(P)  &:= \{ \vv \in \boldsymbol{\Sigma}(P) \quad \text{s.t.} \quad 
\vv_{\tau} = {\boldsymbol 0}  \quad \text{and}  \quad \CC \, \vv = \mathbf{0} 
\quad \text{on $\partial P$} \}
\\
 \boldsymbol{\Psi}_0(P) &:= \{\vv \in \boldsymbol{\Psi}(P)
\quad \text{s.t.} \quad 
\int_{\partial P} \vv \cdot \nn_P \,{\rm d}f = 0 \,, \quad
\vv_{\tau} = {\boldsymbol 0}  \quad \text{and}  \quad \CC \, \vv = \mathbf{0} 
\, \, \text{on $\partial P$} 
 \} \,.
\end{align*}

\begin{remark}
\label{rem:tau_wedge}
Notice that for each face $f \in \partial P$, the vector fields $\vv_f$ and $\vv \wedge \nn_P^f$ are different.
In fact both lie in the plane of the face $f$, but $\vv_f$ is $\pi/2$-rotation in $f$ (with respect to the axes $\boldsymbol{\tau}_1$ and $\boldsymbol{\tau}_2$)  of $\vv \wedge \nn_P^f$. However $\vv_f = {\boldsymbol 0}$ if and only if $\vv \wedge \nn_P^f = {\boldsymbol 0}$. 
For that reason in the definition of $\boldsymbol{\Psi}_0(P)$, we consider a slightly different, but sustantially equivalent, set of homogeneous boundary conditions to that considered in literature \cite{girault-raviart:book, bendali-dominguez-gallic:1985, amrouche-et-al:1998}. 
\end{remark}

Recalling that a sequence is exact if the image of each operator coincides with the kernel of the following one, and that $P$ is contractible, from \eqref{eq:poly_decomposition_3D_grad} and \eqref{eq:poly_decomposition_3D_curl}
it is easy to check that the following sequence is exact \cite{BBMR:2016}: 
\begin{equation}
\label{eq:poly_exact}
\R \, \xrightarrow[]{ \,\quad \text{{$i$}} \quad \,   } \,
\Pk_{n+2}(P) \, \xrightarrow[]{   \quad \text{{$\gr$}} \quad  }\,
[\Pk_{n+1}(P)]^3 \, \xrightarrow[]{ \, \, \, \text{{$\CC$}} \, \, \, }\,
[\Pk_{n}(P)]^3 \, \xrightarrow[]{  \, \, \, \, \text{{$\dd$}} \, \, \, \, }\,
\Pk_{n-1}(P) \, \xrightarrow[]{ \quad 0 \quad }\,
0
\end{equation}
where $i$  denotes the mapping that to every real number $r$ associates the constant function identically equal to $r$ and $0$  is the mapping that to every function associates the number $0$.

The three dimensional de Rham complex  with minimal regularity (in a contractible domain $\Omega$)
is provided by \cite{arnold-falk-winter:2006, demkowicz-et-al:2000}
\begin{equation*}
\R \, \xrightarrow[]{ \,\quad \text{{$i$}} \quad \,   } \,
H^1(\Omega) \, \xrightarrow[]{   \quad \text{{$\gr$}} \quad  }\,
{\boldsymbol H}(\CC, \, \Omega) \, \xrightarrow[]{ \, \, \, \text{{$\CC$}} \, \, \, }\,
{\boldsymbol H}(\dd, \, \Omega) \, \xrightarrow[]{  \, \, \, \, \text{{$\dd$}} \, \, \, \, }\,
L^2(\Omega) \, \xrightarrow[]{ \quad 0 \quad }\,
0 \,.
\end{equation*}
In this paper we consider the de Rham sub-complex with enhanced smoothness  \cite{evans-hughes:2013}
\begin{equation}
\label{eq:exact}
\R \, \xrightarrow[]{ \,\quad \text{{$i$}} \quad \,   } \,
H^1(\Omega) \, \xrightarrow[]{   \quad \text{{$\gr$}} \quad  }\,
\boldsymbol{\Sigma}(\Omega) \, \xrightarrow[]{ \, \, \, \text{{$\CC$}} \, \, \, }\,
[H^1(\Omega)]^3 \, \xrightarrow[]{  \, \, \, \, \text{{$\dd$}} \, \, \, \, }\,
L^2(\Omega) \, \xrightarrow[]{ \quad 0 \quad }
0 \,,
\end{equation}
that is suitable for the Stokes (Navier--Stokes) problem. 
Therefore our goal is to construct conforming (with respect to the decomposition $\Omega_h$) virtual element spaces
\begin{equation}
\label{eq:virtual_spaces}
W_h \subseteq H^1(\Omega) \,, \qquad
\Sg_h  \subseteq \Sg(\Omega) \,, \qquad
\VV_h \subseteq [H^1(\Omega)]^3 \,, \qquad
Q_h \subseteq L^2(\Omega)
\end{equation}
that mimic the complex \eqref{eq:exact}, i.e. are such that 
\begin{equation}
\label{eq:virtual_exact}
\R \, \xrightarrow[]{ \,\quad \text{{$i$}} \quad \,   } \,
W_h \, \xrightarrow[]{   \quad \text{{$\gr$}} \quad  }\,
\Sg_h \, \xrightarrow[]{ \, \, \, \text{{$\CC$}} \, \, \, }\,
\VV_h \, \xrightarrow[]{  \, \, \, \, \text{{$\dd$}} \, \, \, \, }\,
Q_h \, \xrightarrow[]{ \quad 0 \quad }
0 
\end{equation}
is an exact sub-complex of \eqref{eq:exact}.
To the best of our knowledge, no conforming finite elements sub-complex
of \eqref{eq:exact} exists (see for instance \cite{john-linke-merdon-neilan-rebholz:2017}).

%
%


\section{The virtual element spaces}
\label{sec:spaces}

The present section is devoted to the construction of conforming virtual element spaces \eqref{eq:virtual_spaces} that compose the virtual sub-complex \eqref{eq:virtual_exact}.
As we will see, the space $W_h$ consists of the lowest degree three dimensional nodal  VEM space \cite{projectors, BDR:2017}, whereas the spaces $\VV_h$ and $Q_h$ (that are the spaces actually used in the discretization of the problem) are the three dimensional counterparts of the inf-sup stable couple of spaces introduced in \cite{vacca:2018, BLV:2018}. 
Therefore the main novelty of the present section is in the construction of the $\Sg$-conforming space $\Sg_h$. 

In order to facilitate the reading, we present the spaces in the reverse order, from right to left in the sequence \eqref{eq:virtual_exact}. 
In particular, in accordance with \eqref{eq:virtual_exact}, the space $\Sg_h$ will be careful designed to fit $\CC \, \Sg_h \subseteq \VV_h$.

We stress that the readers mainly interested on the virtual elements approximation of the three dimensional Navier--Stokes equation (and not on the virtual de Rham sequence) can skip Subsection \ref{sub:Sg_h}, Subsection \ref{sub:W_h} and Section \ref{sec:derham}. 

\bigskip


One essential idea in the VEM construction is to define suitable (computable) polynomial projections. 
For any $n \in \N$ and each polyhedron/face $\OO$ we introduce the following polynomial projections:
\begin{itemize}

\item the $\boldsymbol{L^2}$\textbf{-projection} $\Pi_n^{0, \OO} \colon L^2(\OO) \to \Pk_n(\OO)$, defined for any $v \in L^2(\OO)$ by
\begin{equation}
\label{eq:P0_k^E}
\int_{\OO} q_n (v - \, {\Pi}_{n}^{0, \OO}  v) \, {\rm d} \OO = 0 \qquad  \text{for all $q_n \in \Pk_n(\OO)$,} 
\end{equation} 
with obvious extension for vector functions $\Pi_n^{0, \OO} \colon [L^2(\OO)]^3 \to [\Pk_n(\OO)]^3$, and tensor functions 
$\boldsymbol{\Pi}_{n}^{0, \OO} \colon [L^2(\OO)]^{3 \times 3} \to [\Pk_{n}(\OO)]^{3 \times 3}$,

\item the $\boldsymbol{H^1}$\textbf{-seminorm projection} ${\Pi}_{n}^{\nabla,\OO} \colon H^1(\OO) \to \Pk_n(\OO)$,  defined for any $v \in H^1(\OO)$ by
\begin{equation}
\label{eq:Pn_k^E}
\left\{
\begin{aligned}
& \int_{\OO} \gr  \,q_n \cdot \gr ( v - \, {\Pi}_{n}^{\nabla,\OO}   v) \, {\rm d} \OO = 0 \qquad  \text{for all $q_n \in \Pk_n(\OO)$,} \\
& \int_{\partial \OO} (v - \,  {\Pi}_{n}^{\nabla, \OO}  v) \,{\rm d}\sigma = 0 \, ,
\end{aligned}
\right.
\end{equation} 
with obvious extension for vector functions $\Pi_n^{\nabla, \OO} \colon [H^1(\OO)]^3 \to [\Pk_n(\OO)]^3$.
\end{itemize}
Let $k \geq 2$ be the polynomial degree of accuracy of the method.
We recall that, in standard finite element fashion, the virtual element spaces are defined element-wise and then are  assembled in such a way the global regularity requirements are satisfied.

\subsection{Scalar $L^2$-conforming space}
\label{sub:Q_h}

We start our construction with the rightmost discrete space $Q_h$ in \eqref{eq:virtual_exact}.
Since we are not requiring any smoothness on $Q_h$, the local space $Q_h(P)$ is simply defined by
\begin{equation*}
Q_h(P) := \Pk_{k-1}(P) \,,
\end{equation*}
having dimension (cf. \eqref{eq:poly_dim})
$\dim(Q_h(P)) = \pi_{k-1, 3}$.
%
The corresponding DoFs are chosen, defining for each $q \in Q_h(P)$
the following linear operators 
\begin{itemize}
\item $\mathbf{D}_{Q}$: the moments up to order $k-1$ of $q$, i.e.,
\[
\int_P q \, p_{k-1} \,{\rm d}P \qquad \text{for any $p_{k-1} \in \Pk_{k-1}(P)$.}
\]
\end{itemize}
The global space is given by
\begin{equation}
\label{eq:Q_h}
Q_h := \{ q \in L^2(\Omega) \quad \text{s.t.} \quad q_{|P} \in Q_h(P) \quad \text{for all $P \in \Omega_h$} \} \,. 
\end{equation}
It is straightforward to see that the dimension of $Q_h$ is 
\begin{equation}
\label{eq:dim_Q_h}
\dim(Q_h) = \pi_{k-1, 3} \,L_P \,.
\end{equation}


\subsection{Vector $H^1$-conforming VEM space}
\label{sub:V_h}

The subsequent space in the de Rham complex \eqref{eq:virtual_exact} is the vector-valued $H^1$-conforming virtual  element space $\VV_h$. 
The construction of $\VV_h$ has to combine two main ingredients: 
\begin{itemize}
\item to define a 3-d version of the space \cite{BLV:2018} that fits the conformity requirement (that definition follows the guidelines of Appendix of reference \cite{BLV:2017});
\item ``to play'' with the enhanced technique \cite{projectors} in order to achieve the computability of the polynomial projections stated in Proposition \ref{prp:proj}.
\end{itemize}
We first consider on each face $f$ of the element $P$, the face space
\begin{equation}
\label{eq:Bhat_h^n}
\begin{aligned}
\widehat{\B}_k(f):= \biggl\{
v \in H^1(f)   \, \, \, \, \text{s.t.} \, \, \, \,
& v_{|\partial f} \in C^0(\partial f) \,,
\quad
{v}_{|e} \in \Pk_k(e) \quad \text{for all $e \in \partial f$,}
\\
& \Delta_f \, v \in \Pk_{k+1}(f) \,, 
\\
& \left(  v - \Pi_k^{\nabla, f} v , \, \widehat{p}_{k+1} \right)_f  = 0\quad 
\text{for all $\widehat{p}_{k+1} \in  \widehat{\Pk}_{k+1 \setminus k-2}(f)$} 
\biggr \}
\end{aligned}
\end{equation}
and the boundary space
\begin{equation*}
\label{eq:Bhatf}
\widehat{\B}_k(\partial P):= \left\{ v \in C^0(\partial P) \quad \text{such that} \quad
v_{|f} \in \widehat{\B}_k(f) 
\quad \text{for any $f \in \partial P$} 
\right\}
\end{equation*}
that is a modification of the standard boundary nodal VEM \cite{BBMR:2016}. Indeed the  ``super-enhanced'' constraints (the last line in the definition \eqref{eq:Bhat_h^n})
are needed to exactly compute the polynomial projection $\Pi_{k+1}^{0, f}$ (see Proposition \ref{prp:proj}).

On the polyhedron $P$ we define the virtual element space $\VV_h(P)$
\begin{equation}
\label{eq:V_h^P}
\begin{aligned}
\VV_h(P) := \biggl\{  
\vv \in [H^1(P)]^3 \, \, \, \text{s.t.}  \, \, \,
& \vv_{|\partial P} \in [\widehat{\B}_k(\partial P)]^3 \,,
\\
&  \biggl\{
\begin{aligned}
& \boldsymbol{\Delta}    \vv  +  \nabla s \in \xx \wedge [\Pk_{k-1}(P)]^3,  \\
& {\rm div} \, \vv \in \Pk_{k-1}(P),
\end{aligned}
\biggr. \quad  \text{ for some $s \in  L^2_0(P)$} 
\\
\biggl.
&  \left(  \vv - \Pi_k^{\nabla, P} \vv  , \, \xx \wedge \widehat{\pp}_{k-1} \right)_P  = 0 \quad
\text{for all $\widehat{\pp}_{k-1} \in  [\widehat{\Pk}_{k-1 \setminus k-3}(P)]^3$}
\biggr \}.
\end{aligned}
\end{equation}
The definition above is the 3-d counterpart of the virtual elements \cite{BLV:2018}, in particular we remark that the enhancing constraints (the last line in \eqref{eq:V_h^P}) are necessary to achieve the computability of the $L^2$-projection $\Pi_k^{0, P}$(see Proposition \ref{prp:proj}).
Moreover, notice that the space $\VV_h(P)$ contains $[\Pk_k(P)]^3$ and this will guarantee the good approximation property of the space (cf. 
Theorem \ref{thm:u}).

\begin{proposition}
\label{prp:V_h^E_dofs}
The dimension of $\VV_h(P)$ is given by
\begin{equation*}
\dim(\VV_h(P)) = 
3 \, \ell_V + 3 \, (k-1)\, \ell_e + 3 \, \pi_{k-2, 2} \, \ell_f + 3 \, \pi_{k-2, 3}  \,.
\end{equation*}
Moreover, the following linear operators $\mathbf{D}_{\VV}$, split into five subsets  constitute a set of DoFs for $\VV_h(P)$:
\begin{itemize}
\item $\mathbf{D^1}_{\VV}$:  the values of $\vv$ at the vertexes of the polyhedron $P$,
\item $\mathbf{D^2}_{\VV}$: the values of $\vv$ at $k-1$ distinct points of every edge $e$ of the polyhedron $P$,
\item $\mathbf{D^3}_{\VV}$: the face moments of $\vv$ 
(split into normal and tangential components)
\begin{equation*}
\int_f (\vv \cdot \nn_P^f) \, p_{k-2} \, {\rm d}f \,,  \qquad
\int_f (\vv \cdot \boldsymbol{\tau}_1^f) \, p_{k-2} \, {\rm d}f \,, \qquad
\int_f (\vv \cdot \boldsymbol{\tau}_2^f) \, p_{k-2} \, {\rm d}f \,,
\end{equation*}
for all $p_{k-2} \in \Pk_{k-2}(f)$,
\item $\mathbf{D^4}_{\VV}$: the volume moments of $\vv$ 
\[
\int_P \vv \cdot (\xx \wedge \pp_{k-3})\, {\rm d}P \qquad 
\text{for all $\pp_{k-3} \in [\Pk_{k-3}(P)]^3$,}
\]
\item $\mathbf{D^5}_{\VV}$: the volume moments of ${\rm div} \,\vv$ 
\[
\int_P ({\rm div} \,\vv) \, \widehat{p}_{k-1} \, {\rm d}P \qquad \text{for all $\widehat{p}_{k-1} \in \widehat{\Pk}_{k-1 \setminus 0}(P)$.}
\] 
\end{itemize}
\end{proposition}

\begin{proof}
We only sketch the proof since it follows the guidelines  of Proposition 3.1 in \cite{vacca:2018} for the analogous 2-d space.
First of all, recalling \eqref{eq:poly_dim} and polynomial decomposition \eqref{eq:poly_decomposition_3D_grad}, simple computations yield
\begin{equation}
\label{eq:D_V}
\begin{gathered}
\texttt{number} (\mathbf{D^1}_{\VV})  = 3 \, \ell_V \,,
\qquad 
\texttt{number} (\mathbf{D^2}_{\VV})  = 3 \, (k-1) \,\ell_e \,,
\qquad
\texttt{number} (\mathbf{D^3}_{\VV})  = 3 \, \pi_{k-2, 2} \, \ell_f \,,
\\
\texttt{number} (\mathbf{D^4}_{\VV})  = 3 \, \pi_{k-2, 3} - \pi_{k-1, 3} + 1 \,,
\qquad
\texttt{number} (\mathbf{D^5}_{\VV})  = \pi_{k-1, 3} -1 \,,
\end{gathered}
\end{equation}
and therefore
\[
\texttt{number} (\mathbf{D}_{\VV}) = 3 \, \ell_V + 3 \, (k-1)\, \ell_e + 3 \, \pi_{k-2, 2} \, \ell_f + 3 \, \pi_{k-2, 3}  \,.
\]
Now employing Proposition 2 and Remark 5 in \cite{projectors}, it can be shown that the DoFs $\mathbf{D^1}_{\VV}$, $\mathbf{D^2}_{\VV}$, $\mathbf{D^3}_{\VV}$ are unisolvent for the space $[\widehat{\B}_k(\partial P)]^3$.
Therefore it holds that
\begin{equation}
\label{eq:dofs3}
\dim([\widehat{\B}_k(\partial P)]^3)  = 
3 \, \ell_V + 3 \, (k - 1) \, \ell_e + 3 \, \pi_{k-2, 2} \, \ell_f \,,
\end{equation}
which in turn implies (recalling \eqref{eq:V_h^P})
\[
\dim(\VV_h(P)) \geq \texttt{number} (\mathbf{D}_{\VV})
 \,.
\]
Now the result follows by proving that $\mathbf{D}_{\VV}(\vv) = {\boldsymbol 0}$ implies that $\vv$ is identically zero, that can be shown first works on $\partial P$ and then inside $P$.
As a consequence the linear operators $\mathbf{D}_{\VV}$ are unisolvent for $\VV_h$ and in particular
$\dim(\VV_h(P)) = \texttt{number} (\mathbf{D}_{\VV})$.
\end{proof}

The global space $\VV_h$ is defined by gluing the local spaces with the obvious associated sets of global DoFs:
\begin{equation}
\label{eq:V_h}
\VV_h := \{\vv \in [H^1(\Omega)]^3 \quad \text{s.t.} \quad \vv_{|P} \in \VV_h(P) \} \,.
\end{equation}
The dimension of $\VV_h$ is given by
\begin{equation}
\label{eq:dim_V_h}
\dim(\VV_h) = 
3 \, L_V + 3 \, (k-1)\, L_e + 3 \, \pi_{k-2, 2} \, L_f + 3 \, \pi_{k-2, 3} \, L_P  \,.
\end{equation}
We also consider the  discrete  local kernel 
\begin{equation*}
\ZZ_h(P) := \left \{
\vv \in \VV_h(P) \quad \text{s.t.} \quad \int_{P} \dd \, \vv \, q \,{\rm dP}= 0 \quad \text{for  all $q \in Q_h(P)$} 
\right\}\,,
\end{equation*}
and the corresponding  global version 
\begin{equation}
\label{eq:Z_h}
\ZZ_h := \left\{ \vv \in \VV_h \quad \text{s.t.} \quad \int_{\Omega} \dd \, \vv \, q \,{\rm d\Omega}= 0 \quad \text{for  all $q \in Q_h$} \right\} \,.
\end{equation}
A crucial observation is that, extending to the 3-d case the result in  \cite{BLV:2017},  the proposed discrete spaces \eqref{eq:Q_h} and \eqref{eq:V_h} are such that
$
\dd \, \VV_h \subseteq Q_h
$.
As a consequence the considerable kernel inclusion holds
\begin{equation}
\label{eq:kernel_inclusion}
\ZZ_h \subseteq \ZZ \,.
\end{equation}
The inclusion here above and explicit computations (cf. \eqref{eq:D_V}) yield that
\begin{equation*}
\dim(\ZZ_h) = 
3 \, L_V + 3 \, (k-1)\, L_e + 3 \, \pi_{k-2, 2} \, L_f + (3 \, \pi_{k-2, 3} - \pi_{k-1, 3}) \, N_P \,.
\end{equation*}
The notable property \eqref{eq:kernel_inclusion} leads to a series of important advantages, as explored in 
\cite{linke-merdon:2016, john-linke-merdon-neilan-rebholz:2017, BLV:2018}. 

\begin{remark}
\label{rm:projectionPD}
In the third line of Definition \eqref{eq:V_h^P} the $H^1$-seminorm projection $\Pi^{\nabla, P}_k$ can be actually replaced by any polynomial projection $\Pi^{P}_k$ that is computable on the basis of the DoFs $\mathbf{D}_{\VV}$ (in the sense of Proposition \ref{prp:proj}).
This change clearly  propagates throughout the rest of the analysis (see Definitions \eqref{eq:Sg_h^P} and \eqref{eq:Sg_h^Pnew}). 
An analogous observation holds also for the operator $\Pi^{\nabla, f}_k$ in the third line of Definition \eqref{eq:Bhat_h^n}. The present remark allows to make use of computationally cheaper projections, as done in the numerical tests of Section \ref{sec:num_test}.
\end{remark}

\subsection{Vector $\Sg$-conforming VEM space}
\label{sub:Sg_h}

In the present subsection we consider the construction of the $\Sg$-conforming virtual space $\Sg_h$ in \eqref{eq:virtual_exact}. As mentioned before, this brick constitutes the main novelty in the foundation of the virtual de Rham sequence \eqref{eq:virtual_exact}.
The core ideas in building such space are the following:
\begin{itemize}
\item the space is careful designed to satisfy $\CC \, \Sg_h = \ZZ_h$;
\item the DoFs are conveniently chosen in order to have a direct correspondence   between the $\CC$ of the Lagrange-type basis functions of $\Sg_h$ and the Lagrange basis functions of $\VV_h$; 
\item the boundary space and the boundary DoFs are picked in accordance with the global conformity requirements ensuing
from the regularity of the space $\Sg$.
\end{itemize}
We start by introducing on each face $f \in \partial P$  the face space 
\begin{equation}
\label{eq:Sf}
\begin{aligned}
\boldsymbol{S}_k(f) := \biggl\{
\ssigma \in {\boldsymbol H}(\dd_f, \, f) \cap {\boldsymbol H}(\rr_f, \, f)
\quad \text{s.t.} \quad
&  ({\ssigma} \cdot \tf_e)_{|e} \in \Pk_0(e) &
 \text{$\forall e \in \partial f$,} 
\\
& \dd_f \, \ssigma = 0 \,,
\\ 
& \rr_f \, \ssigma \in \widehat{\B}_k(f) 
&  & \quad \biggr\} \,,
\end{aligned}
\end{equation}
and the boundary space
\begin{equation}
\label{eq:S}
\begin{aligned}
{\boldsymbol S}_k(\partial P):= \biggl\{
\ssigma \in [L^2(\partial P)]^3 \quad \text{s.t.} \quad
&\ssigma_{\tau} \in {\boldsymbol S}_k(f)
&\text{for any $f \in \partial P$,}
\\
& ({\ssigma_{f_1}} \cdot \tf_e)_{|e} = ({\ssigma_{f_2}} \cdot \tf_e)_{|e}  
&\forall 
e \subseteq \partial f_1 \cap \partial f_2 \,, &\quad f_1, f_2 \in \partial P 
\biggr \} \,.
\end{aligned}
\end{equation}
%
%
Concerning the  differential problem in definition \eqref{eq:Sf}, 
we recall that on simply connected polygon $f$, given two sufficiently regular functions ${\boldsymbol g}$ and ${\boldsymbol h}$ defined on $f$ and a sufficiently regular function $\omega$ defined on $\partial f$,  the problem
\begin{equation}
\label{eq:hdivhrot}
\left \{
\begin{aligned}
& \text{find $\ssigma   \in {\boldsymbol H}(\dd_f, \, f) \cap {\boldsymbol H}(\rr_f, \, f)$ s.t.}
\\
&
\begin{aligned}
& \dd_f \, \ssigma =  {\boldsymbol g} \qquad &\text{in $f$,} \\
& \rr_f \, \ssigma =  {\boldsymbol h} \qquad &\text{in $f$,} \\
& \ssigma \cdot \tf_f = \omega \qquad &\text{on $\partial f$,} 
\end{aligned}
\end{aligned}
\right .
\end{equation}
is well posed if and only if, in accordance with \eqref{eq:compatibility_rr_tf}, the following holds 
\begin{equation}
\label{eq:compatibility_rr_tf_prb}
\int_{f} {\boldsymbol h} \, {\rm d}f = \int_{\partial f} \omega \, {\rm d}s \,.
\end{equation}
On the polyhedron $P$ we define the virtual space: 
\begin{equation}
\label{eq:Sg_h^P}
\begin{aligned}
\Sg_h(P) := \biggl\{  
&\ffi \in \boldsymbol{\Psi}(P) \quad \text{s.t.} \quad
 \ffi_{|\partial P}  \in {\boldsymbol S}_k(\partial P) \, , 
\quad
\int_{\partial P}  \ffi \cdot \nn_P \, {\rm d}f = 0 \,, 
\biggr.
\\
& \qquad (\CC \, \ffi)_{|\partial P}   \in [\widehat{\B}_k(\partial P)]^3 \,,
\\
& \qquad
\int_P \dl \, \ffi \cdot \dl \, \ppsi \, {\rm d}P = 
\int_P  \widetilde{\pp}_{k-1} \cdot \ppsi \, {\rm d}P,
\qquad
\text{$\forall \ppsi \in \boldsymbol{\Psi}_0(P)$,}
\\ 
& \qquad \qquad \qquad \qquad \qquad
\text{for some  $\widetilde{\pp}_{k-1} \in [\Pk_{k-1}(P)]^3 \cap \ZZ(P)$ 
,}
\\
\biggl.
& \qquad \left(  \CC \, \ffi - \Pi_k^{\nabla, P} \CC \, \ffi  , \, \xx \wedge \widehat{\pp}_{k-1} \right)_P  = 0\quad 
\text{$\forall \, \widehat{\pp}_{k-1} \in  [\widehat{\Pk}_{k-1 \setminus k-3}(P)]^3$}
\biggr \}.
\end{aligned}
\end{equation}
We stress that the variational problem stated in \eqref{eq:Sg_h^P} is coupled with the non homogeneous version of the boundary conditions in \cite{girault-raviart:book, amrouche-et-al:1998}. 
In fact, in order to force $\Sg$-conforming regularity, for any function $\ffi \in \Sg_h(P)$,  we need to prescribe $\CC \, \ffi$ and $\ffi_{\tau}$ on $\partial P$.
We address the well-posedness of the biharmonic problem in definition \eqref{eq:Sg_h^P} in the Appendix.

Note that, in accordance with the target $\CC \, \Sg_h  \subseteq \VV_h$, the second and the last line in definition \eqref{eq:Sg_h^P} are the $\CC$ version of the first and last line in definition \eqref{eq:V_h^P}. Whereas we will see that $\CC$ of the solutions of the biharmonic problem  in  \eqref{eq:Sg_h^P} are solutions to the Stokes problem in \eqref{eq:V_h^P} (see Proposition \ref{prp:inclusion2}).

\begin{proposition}
\label{prp:Sg_h}
The dimension of $\Sg_h(P)$ is given by
\begin{equation*}
\dim(\Sg_h(P)) = 
3 \, \ell_V + (3 k-2)\, \ell_e + (3 \, \pi_{k-2, 2} -1)\, \ell_f + 3 \, \pi_{k-2, 3} - \pi_{k-1, 3} + 1 \,.
\end{equation*}

Moreover, the following linear operators $\mathbf{D}_{\Sg}$, split into five subsets  constitute a set of DoFs for $\Sg_h(P)$:
\begin{itemize}
\item $\mathbf{D^1}_{\Sg}$:  the values of $\CC \, \ffi$ at the vertexes of the polyhedron $P$,
\item $\mathbf{D^2}_{\Sg}$: the values of $\CC \, \ffi$ at $k-1$ distinct points of every edge $e$ of the polyhedron $P$,
\item $\mathbf{D^3}_{\Sg}$: the face moments of $\CC \, \ffi$ 
(split into normal and tangential components)
\begin{gather*}
\int_f (\CC \, \ffi \cdot \nn_P^f) \, \widehat{p}_{k-2} \, {\rm d}f \,,  \qquad
\text{for all $\widehat{p}_{k-2}  \in \widehat{\Pk}_{k-2 \setminus 0}(f)$,}
\\
\int_f (\CC \, \ffi \cdot \boldsymbol{\tau}_1^f) \, p_{k-2} \, {\rm d}f \,, \qquad
\int_f (\CC \, \ffi \cdot \boldsymbol{\tau}_2^f) \, p_{k-2} \, {\rm d}f \,,
\qquad \text{for all $p_{k-2} \in \Pk_{k-2}(f)$},
\end{gather*}
\item $\mathbf{D^4}_{\Sg}$: the volume moments of $\CC \, \ffi$ 
\[
\int_P \CC \, \ffi \cdot (\xx \wedge \pp_{k-3})\, {\rm d}P \qquad 
\text{for all $\pp_{k-3} \in [\Pk_{k-3}(P)]^3$,}
\]
\item $\mathbf{D^5}_{\Sg}$:  the edge mean value of $\ffi \cdot \tf_e$, i.e. 
\[
\frac{1}{|e|} \, \int_e \ffi \cdot \tf_e \,{\rm d}s \,.
\]
\end{itemize}
\end{proposition}

\begin{proof}
We start the proof counting the number of the linear operators $\mathbf{D}_{\Sg}$. Using similar  computations as in \eqref{eq:D_V} we have:
\begin{equation*}
\begin{gathered}
\texttt{number} (\mathbf{D^1}_{\Sg})  = 3 \, \ell_V \,,
\qquad
\texttt{number} (\mathbf{D^2}_{\Sg})  = 3 \, (k-1) \,\ell_e \,,
\qquad
\texttt{number} (\mathbf{D^3}_{\Sg})  = (3 \, \pi_{k-2, 2} -1) \, \ell_f \,,
\\
\texttt{number} (\mathbf{D^4}_{\Sg})  = 3 \, \pi_{k-2, 3} - \pi_{k-1, 3} + 1 \,,
\qquad
\texttt{number} (\mathbf{D^5}_{\Sg})  =  \ell_e\,,
\end{gathered}
\end{equation*}
and thus
\begin{equation}
\label{eq:numberDSg}
\texttt{number} (\mathbf{D}_{\Sg}) = 3 \, \ell_V + (3 k-2)\, \ell_e + (3 \, \pi_{k-2, 2} -1)\, \ell_f + 3 \, \pi_{k-2, 3} - \pi_{k-1, 3} + 1 \,.
\end{equation}
For sake of simplicity, we prove that $\mathbf{D}_{\Sg}$ constitutes a set of DoFs for the non-enhanced space associated with $\Sg_h(P)$, i.e. the space 
$\widetilde{\Sg}_h(P)$ obtained by dropping the last line in \eqref{eq:Sg_h^P} (the enhanced constraints) and by taking in the biharmonic system $\widetilde{\pp}_{k-3} \in [\Pk_{k-3}(P)]^3 \cap \ZZ(P)$, i.e. the space
\[
\begin{aligned}
\widetilde{\Sg}_h(P) := \biggl\{  
&\ffi \in \boldsymbol{\Psi}(P) \quad \text{s.t.} \quad
 \ffi_{|\partial P}  \in {\boldsymbol S}_k(\partial P) \, , 
\quad
\int_{\partial P}  \ffi \cdot \nn_P \, {\rm d}f = 0 \,, 
\biggr.
\\
& \qquad (\CC \, \ffi)_{|\partial P} \in [\widehat{\B}_k(\partial P)]^3\,,
\\
& \qquad
\int_P \dl \, \ffi \cdot \dl \, \ppsi \, {\rm d}P = 
\int_P  \widetilde{\pp}_{k-3} \cdot \ppsi \, {\rm d}P,
\qquad
\text{$\forall \ppsi \in \boldsymbol{\Psi}_0(P)$,}
\\ 
& \qquad \qquad \qquad \qquad \qquad
\text{for some  $\widetilde{\pp}_{k-3} \in [\Pk_{k-3}(P)]^3 \cap \ZZ(P)$}
\biggr \}\,.
\end{aligned}
\]
Once the proof for $\widetilde{\Sg}_h(P)$ is given, the extension to the original space $\Sg_h(P)$ easily follows by employing standard techniques for VEM enhanced spaces (see \cite{projectors} and Proposition 5.1 in \cite{BMV:2018}).

Employing Theorem \ref{thm:wellposed}, given
\begin{itemize}
\item  $\widetilde{\pp}_{k-3} \in [\Pk_{k-3}(P)]^3 \cap \ZZ(P)$,
\item  ${\boldsymbol g} \in  [\widehat{\B}_k(\partial P)]^3$,
\item  ${\boldsymbol h} \in {\boldsymbol S}_k(\partial P)$ satisfying the compatibility condition (cf. \eqref{eq:compatibility_CC_rr})
\begin{equation}
\label{eq:dofs0}
{\boldsymbol g} \cdot \nn_P^f = \rr_f \, {\boldsymbol h}_{\tau} \quad \text{on  any $f \in \partial P$,}
\end{equation}
\end{itemize}
there exists a unique function $\ffi \in \boldsymbol{\Psi}(P)$ such that 
\[
\left \{
\begin{aligned}
&\int_P \dl \, \ffi \cdot \dl \, \ppsi \, {\rm d}P = 
\int_P  \widetilde{\pp}_{k-3} \cdot \ppsi \, {\rm d}P,
\qquad
&\text{for all $\ppsi \in \boldsymbol{\Psi}_0(P)$,}
\\
&\int_{\partial P} \ffi \cdot \nn_P \,{\rm d}f = 0 \,,
\\
&\ffi_{\tau}  =  {\boldsymbol h}_{\tau}  \qquad &\text{on $\partial P$,}
\\
&\CC \, \ffi = {\boldsymbol g}  \qquad &\text{on $\partial P$.}
\end{aligned}
\right .
\] 
Therefore
\begin{equation}
\label{eq:dofs1}
\begin{split}
\dim(\widetilde{\Sg}_h(P))  
&= \dim([\Pk_{k-3}(P)]^3 \cap \ZZ(P)) + 
  \dim([\widehat{\B}_k(\partial P)]^3) +
  \dim({\boldsymbol S}_k(\partial P)) 
   - \dim(\widehat{\B}_k(f)) \, \ell_f
\end{split}  
\end{equation}
where
the last term $(-\dim(\widehat{\B}_k(f)) \, \ell_f)$ ensues from the compatibility condition \eqref{eq:dofs0}.
We calculate the addenda in the right hand side of \eqref{eq:dofs1}.
Regarding the first term in \eqref{eq:dofs1}, we preliminary note
that the following characterization ensues from the 
exact sequence \eqref{eq:poly_exact} and polynomial decomposition \eqref{eq:poly_decomposition_3D_grad}
\begin{equation}
\label{eq:z_CC}
[\Pk_{k-3}(P)]^3 \cap \ZZ(P) = \CC \, \left( [\Pk_{k-2}(P)]^3\right) =
\CC \left( \xx \wedge [\Pk_{k-3}(P)]^3 \right) \,.
\end{equation}
Employing again the exact sequence \eqref{eq:poly_exact}, $\CC$ restricted to $\left( \xx \wedge [\Pk_{k-3}(P)]^3 \right)$ is actually an isomorphism, therefore from \eqref{eq:z_CC} and \eqref{eq:D_V} follows that
\begin{equation}
\label{eq:dofs2}
\dim([\Pk_{k-3}(P)]^3 \cap \ZZ(P))  = 
\dim \left( \xx \wedge [\Pk_{k-3}(P)]^3\right)=
3 \, \pi_{k-2, 3} - \pi_{k-1, 3} + 1 \,.
\end{equation}
From definitions \eqref{eq:Sf} and \eqref{eq:S} and since problem \eqref{eq:hdivhrot} is well-posed, direct computations
yield
\begin{equation}
\label{eq:dofs4}
\dim({\boldsymbol S}_k(\partial P)) = \ell_e + (\dim(\widehat{\B}_k(f)) - 1) \,
\ell_f 
\end{equation}
where the $-1$ in the formula above is due to the compatibility condition \eqref{eq:compatibility_rr_tf_prb}.

Collecting  \eqref{eq:dofs2}, \eqref{eq:dofs3} and \eqref{eq:dofs4} in \eqref{eq:dofs1} (compare with \eqref{eq:numberDSg}) we get
\[
\dim(\widetilde{\Sg}_h(P))  = \texttt{number} (\mathbf{D}_{\Sg}) \,.
\] 
Having proved that $\texttt{number} (\mathbf{D}_{\Sg})$ is  equal to $\dim(\widetilde{\Sg}_h(P))$, in order to validate that the linear operators $\mathbf{D}_{\Sg}$ constitute a set of DoFs for $\widetilde{\Sg}_h(P)$ we have to check that they are unisolvent.
Let $\ffi \in \widetilde{\Sg}_h(P)$ such that $\mathbf{D}_{\Sg}(\ffi) = {\boldsymbol 0}$, we need to show that $\ffi$ is identically zero.
It is straightforward that 
$\mathbf{D^1}_{\Sg}(\ffi) = \mathbf{D^2}_{\Sg}(\ffi) = {\boldsymbol 0}$ implies  
\begin{equation}
\label{eq:unisolvency1}
(\CC \, \ffi)_{|\partial f} = {\boldsymbol 0} \qquad \text{for any $f \in \partial P$.}
\end{equation}
Recalling the well known results for nodal boundary spaces \cite{BBMR:2016}, it is quite obvious to check that \eqref{eq:unisolvency1} $\mathbf{D^3}_{\Sg}(\ffi) = {\boldsymbol 0}$ implies 
\begin{equation*}
\label{eq:unisolvency2}
(\CC \, \ffi)_{f} = {\boldsymbol 0} \qquad \text{for any $f \in \partial P$.}
\end{equation*}
In order to get also the normal component of $(\CC \, \ffi)_{|f}$ equal to zero, based on  $\mathbf{D^3}_{\Sg}(\ffi) = {\boldsymbol 0}$, 
it is sufficient to observe that the compatibility conditions \eqref{eq:compatibility_CC_rr} and \eqref{eq:compatibility_rr_tf}
give
\begin{equation}
\label{eq:unisolvency2a}
\begin{split}
\int_f \CC \, \ffi \cdot \nn_P^f \, {\rm d}f &= 
\int_{f} \rr_f \, \ffi_{\tau} \,{\rm d}f  =
\int_{\partial f} \ffi \cdot \tf_f \,{\rm d}s 
\\
&=
\sum_{e \in \partial f} \int_{ e} \ffi \cdot \tf_f^e \,{\rm d}s =
\sum_{e \in \partial f} |e| \, \mathbf{D^5}_{\Sg, e}(\ffi) \, \tf_e \cdot \tf_f^e 
\qquad \text{for any $f \in \partial P$.}
\end{split}
\end{equation}
that is equal to 0 since  $\mathbf{D^5}_{\Sg}(\ffi) = {\boldsymbol 0}$.
Therefore we have proved that 
\begin{equation}
\label{eq:unisolvency3}
\CC \, \ffi = {\boldsymbol 0} \qquad \text{on $\partial P$.}
\end{equation}
Moreover, being $\mathbf{D^5}_{\Sg}(\ffi) = {\boldsymbol 0}$,
from \eqref{eq:unisolvency3} and \eqref{eq:compatibility_CC_rr} and definition \eqref{eq:Sf} 
for any $f \in \partial P$ we infer
\[
(\ffi_f \cdot \tf_e)_{|e} = 0 \quad \text{for any $e \in \partial f$,}
\qquad 
\dd_f \, \ffi_{\tau} = 0 \quad \text{and} \quad
\rr_f \, \ffi_{\tau} = 0 \quad \text{on $f$,} 
\] 
and thus, being \eqref{eq:hdivhrot} well-posed, we obtain
\begin{equation}
\label{eq:unisolvency4}
\ffi_{\tau} = {\boldsymbol 0} \qquad \text{for any $f \in \partial P$.}
\end{equation}
Finally, by definition of $\widetilde{\Sg}_h(P)$, there exists $\widetilde{\pp}_{k-3} \in [\Pk_{k-3}(P)]^3 \cap \ZZ(P)$ such that
\[
\int_P \dl \, \ffi \cdot \dl \, \ppsi \, {\rm d}P = 
\int_P  \widetilde{\pp}_{k-3} \cdot \ppsi \, {\rm d}P
\qquad
\text{for all $\ppsi \in \boldsymbol{\Psi}_0(P)$.}
\]
Therefore, being $\ffi \in \boldsymbol{\Psi}_0(P)$  (cf. \eqref{eq:unisolvency3} and \eqref{eq:unisolvency4}) we infer
\begin{equation*}
\label{eq:unisolvency6}
\begin{aligned}
\|\dl \, \ffi\|_{0, P}^2 & =
\int_P \dl \, \ffi \cdot \dl \, \ffi \, {\rm d}P = 
\int_P  \widetilde{\pp}_{k-3} \cdot \ffi \, {\rm d}P
\\ 
& = \int_P  \CC \, (\xx \wedge {\boldsymbol q}_{k-3}) \cdot \ffi \, {\rm d}P
\quad & \text{(characterization \eqref{eq:z_CC})}
\\
& =  \int_P  \xx \wedge {\boldsymbol q}_{k-3} \cdot  \CC \, \ffi \, {\rm d}P
\quad & \text{(integration by parts + \eqref{eq:unisolvency4})}
\end{aligned}
\end{equation*}
and thus, since $\mathbf{D^4}_{\Sg}(\ffi) = {\boldsymbol 0}$, we obtain
$
\|\dl \, \ffi\|_{0, P}^2 = 0
$.
Now the proof follows by the fact that 
$\|\dl \,\cdot\|_{0, P}$ is a norm on $\boldsymbol{\Psi}_0(P)$ (see Lemma 5.2 in \cite{girault-raviart:book} and \eqref{eq:norm_equivalence}).
\end{proof}

\begin{remark}
\label{rm:Sg_h}
A careful inspection of Theorem \ref{thm:wellposed} (see also Remark 5.1 in \cite{girault-raviart:book} and \cite{bendali-dominguez-gallic:1985}) reveals that the space \eqref{eq:Sg_h^P} admits the equivalent formulation
\begin{equation}
\label{eq:Sg_h^Pnew}
\begin{aligned}
\Sg_h(P) := \biggl\{  
&\ffi \in \boldsymbol{\Psi}(P) \quad \text{s.t.} \quad
 \ffi_{|\partial P}  \in {\boldsymbol S}_k(\partial P) \,, 
\quad
 (\CC \, \ffi)_{|\partial P}   \in [\widehat{\B}_k(\partial P)]^3 \,,
\\
& \qquad
\dl^2 \, \ffi  \in [\Pk_{k-1}(P)]^3 \cap \ZZ(P) \,,
\\
& \qquad
\dd \, \ffi = 0 \,,
\\
\biggl.
&  \qquad
\left(  \CC \, \ffi - \Pi_k^{\nabla, P} \CC \, \ffi  , \, \xx \wedge \widehat{\pp}_{k-1} \right)_P  = 0\quad 
\text{$\forall \, \widehat{\pp}_{k-1} \in  [\widehat{\Pk}_{k-1 \setminus k-3}(P)]^3$}
\biggr \} \,.
\end{aligned}
\end{equation}
\end{remark}

The global space $\Sg_h$ is defined by collecting the local spaces $\Sg_h(P)$, i.e.
\begin{equation}
\label{eq:Sg_h}
\Sg_h := \{\ffi \in \Sg(\Omega) \quad \text{s.t.} \quad \ffi_{|P} \in \Sg_h(P) \} \,.
\end{equation}
The global set  of DoFs is the global counterpart of $\mathbf{D}_{\Sg}$, in particular the choice of DoFs $\mathbf{D}_{\Sg}$ establishes the conforming property
$\CC \, \Sg_h \subseteq [H^1(\Omega)]^3$.
The dimension of $\Sg_h$ is given by
\begin{equation*}
\dim(\Sg_h) = 
3 \, L_V + (3 k-2)\, L_e + (3 \, \pi_{k-2, 2} -1) \, L_f + (3 \, \pi_{k-2, 3} - \pi_{k-1, 3} + 1) \,L_P \,.
\end{equation*}

\subsection{Scalar $H^1$-conforming VEM space}
\label{sub:W_h}

In the present section we briefly define the  $H^1$-conforming space $W_h$ in the virtual complex \eqref{eq:virtual_exact}. The space $W_h$ consists of low order nodal VEM  \cite{BBMR:2016}.

We first introduce the low order boundary space
\begin{equation}
\label{eq:B_h^n}
\B_1(f):= \{
v \in H^1(f) \quad \text{s.t.} \quad
v_{|\partial f} \in C(\partial f) \,, \quad
{v}_{|e} \in \Pk_1(e) \quad \text{for all $e \in \partial f$}\,, \quad
\Delta_f \, v = 0   
\}\,,
\end{equation}
and then we consider the VEM space on the polyhedron $P$
\begin{equation}
\label{eq:W_h^P}
W_h(P) := 
\{
v \in  H^1(P)
\, \, \, \text{s.t.} \, \, \, 
v_{|\partial P} \in C^0(\partial P)\,, \, \, \, 
v_{|f} \in \B_1(f) \quad \text{for any $f \in \partial P$,}  \, \, \, 
\Delta \, v = 0 
\}\,,
\end{equation}
with the associated set of DoFs:
\begin{itemize}
\item $\mathbf{D}_{W}$:  the values of $v$ at the vertexes of the polyhedron $P$.
\end{itemize}
It is straightforward to see that the dimension of $W_h(P)$ is given by
$
\dim(W_h(P)) = \ell_V
$.

The global space is obtained by collecting the local spaces
\begin{equation}
\label{eq:W_h}
W_h := 
\{
v \in H^1(\Omega) 
\quad \text{s.t.} \quad 
 v_{|P} \in W_h(P)  \quad \text{for all $P \in \Omega_h$} 
\}
\end{equation}
with the obvious  associated DoFs.
The dimension of $W_h$ thus is given by
\begin{equation*}
\dim(W_h) = L_V \,.
\end{equation*}


\section{The virtual elements de Rham sequence}
\label{sec:derham}

The aim of the present section is to show that the set of virtual spaces introduced in Section \ref{sec:spaces} realizes the exact sequence \eqref{eq:virtual_exact}.

\begin{theorem}
\label{thm:exact}
The sequence \eqref{eq:virtual_exact} constitutes an exact complex.
\end{theorem}

The theorem follows by Proposition \ref{prp:inclusion1}, Proposition \ref{prp:inclusion2} and Proposition \ref{prp:inclusion3}, here below, stating that the image of each operator in \eqref{eq:virtual_exact} coincides with the kernel of the following one.

\begin{proposition}
\label{prp:inclusion1}
Let $W_h$ and $\Sg_h$ be the spaces defined in \eqref{eq:W_h} and \eqref{eq:Sg_h} respectively. Then 
\[
\nabla\, W_h =  \{ \ffi \in \Sg_h \quad \text{s.t.} \quad \CC \, \ffi = {\bf 0} \quad \text{in $\Omega$}\} \,.
\]
\end{proposition}

\begin{proof}
Essentially we need to prove that 
\begin{itemize}
\item [${\boldsymbol {(i1)}}$]: for every $w \in W_h$, $\nabla \, w \in \Sg_h$ and $\CC (\nabla \, w) = {\bf 0}$,
\item [${\boldsymbol {(i2)}}$]: for every $\ffi \in \Sg_h$ with $\CC \, \ffi = {\bf 0}$, there exists $w \in W_h$  such that $\nabla \, w = \ffi$.
\end{itemize}
For what concerns the inclusion ${\boldsymbol {(i1)}}$, every $w \in W_h$ clearly satisfies
\[
\nabla \, w \in [L^2(\Omega)]^3 
\quad \text{and} \quad
\CC (\nabla \, w) = {\bf 0} \in [H^1(\Omega)]^3\,.
\]
Therefore we need to verify that $(\nabla \, w)_{|P} \in \Sg_h(P)$ for  any $P \in \Omega_h$. 
Notic0e that the tangential component of $\nabla \, w$ verifies 
\begin{equation}
\label{eq:grad_tan}
(\nabla \, w)_{\tau} = \nabla_f \, w_{\tau} 
\qquad \text{ on each face $f \in \partial P$.} 
\end{equation}
From definition  \eqref{eq:B_h^n} and \eqref{eq:Sf}, 
 for any $f \in \partial P$ we infer
\begin{align*}
& \dd_f ( \nabla_f \, w_{\tau} ) =  0 
\qquad \text{in $f$,}
\qquad \qquad &\text{($\Delta_f \, w_{\tau} = 0$)}
\\
&  \rr_f ( \nabla_f \, w_{\tau} ) = 0 
\qquad \text{in $f$,}
\qquad \qquad &\text{(vector calculs identity)}
\\
& (\nabla \, w \cdot \tf_e)_{|e} = \frac{\partial w}{\partial \tf_e} \in \Pk_0(e) 
\qquad \text{$\forall \, e \in \partial f$,} 
\qquad \qquad &\text{($w_{|e} \in \Pk_1(e)$)}
\end{align*}
that, recalling \eqref{eq:grad_tan}, implies $(\nabla \, w)_{\tau} \in {\boldsymbol S}_k(f)$. 
Moreover $w \in C^0(\partial P)$ entails 
\[
[(\nabla \, w)_{f_1} \cdot \tf_e]_{|e} = [(\nabla \, w)_{f_2} \cdot \tf_e]_{|e} 
\quad \text{for any $e \subseteq \partial f_1 \cap f_2$,}
\]
and thus (cf. definition 
\eqref{eq:S})
\begin{equation}
\label{eq:grad_sk}
(\nabla \, w)_{|\partial P} \in {\boldsymbol S}_k(\partial P) \,.
\end{equation} 
Furthermore definition \eqref{eq:W_h^P} implies
\begin{equation}
\label{eq:grad_sg}
\begin{aligned}
&\int_{\partial P} \nabla \, w \cdot \nn_P \,{\rm d}f =
\int_P \Delta \, w \, {\rm d}P = 0 \,,
&
\qquad &\text{(div. thereom + $\Delta \, w = 0$)}
\\
&\CC ( \nabla \, w) = {\bf 0} 
&\text{in $\overline{P}$,} 
\qquad &\text{(vector calculus identity)}
\\
& \dl (\nabla \, w) = \Gr (\Delta \, w) = {\bf 0} 
&\text{in $P$.}
\qquad &\text{($\Delta \, w = 0$)}
\end{aligned}
\end{equation}
Collecting \eqref{eq:grad_sk} and \eqref{eq:grad_sg} in definition \eqref{eq:Sg_h^P}, we easily obtain $\boldsymbol{(i1)}$.

We prove now the property $\boldsymbol{(i2)}$. 
Consider $\ffi \in \Sg_h$ such that $\CC \, \ffi = {\bf 0}$.
Since \eqref{eq:exact} is an exact sequence, there exists unique (up to constant) $\widetilde{w} \in H^1(\Omega)$ such that $\nabla \, \widetilde{w} =\ffi$.
Therefore for any face $f$ in the decomposition $\Omega_h$, 
the tangential component of $\nabla \, \widetilde{w}$ satisfies 
(cf. definition \eqref{eq:Sf})
\[
(\nabla \, \widetilde{w})_{\tau} = 
\nabla_f \, \widetilde{w}_{\tau} =  \ffi_{\tau} \in [L^2(f)]^2 \quad \text{on $f$.}
\]
Hence on each face $f$ the function $\widetilde{w}$ fulfils
\begin{equation}
\label{eq:gradwt}
\left \{
\begin{aligned}
& (\nabla \, \widetilde{w} \cdot \tf_e)_{|e} = (\ffi \cdot \tf_e)_{|e} \in \Pk_0(e) \quad 
&\text{on any $e \in \partial f$,}
\\
& \widetilde{w}_{\tau} \in H^1(f) 
\quad 
&\text{in $f$.}
\end{aligned}
\right .
\end{equation}
From \eqref{eq:gradwt} it follows that $\widetilde{w}$ restricted to the mesh skeleton is continuous and piecewise linear. 
Thus the function $\widetilde{w}$ is well defined (single valued) on the vertexes of the decomposition $\Omega_h$ and $\mathbf{D}_{\VV}(\widetilde{w})$ makes sense.
Let now
 $w \in W_h$ be the interpolant function of $\widetilde{w}$ in the sense of DoFs, i.e. the function uniquely determined by
\begin{equation}
\label{eq:gradwt1}
\mathbf{D}_{\VV}(w) = \mathbf{D}_{\VV}(\widetilde{w}) \,.
\end{equation}
Inclusion $\boldsymbol{(i1)}$ guarantees that $\nabla \, w \in \Sg_h$. 
Hence, by Proposition \ref{prp:Sg_h}, $w$ realizes $\boldsymbol{(i2)}$
if and only if
$\mathbf{D}_{\Sg}(\nabla \, w) = \mathbf{D}_{\Sg}(\ffi)$.
Being $\CC (\nabla \, w) = \CC \, \ffi = {\bf 0}$, this reduce to verify that 
\[
\mathbf{D^5}_{\Sg}(\nabla \, w) = \mathbf{D^5}_{\Sg}(\ffi) \,.
\] 
For any edge $e$ in the decomposition $\Omega_h$, we denote with $\nu_2$ and $\nu_1$ the two endpoints of  $e$, with $\tf_e$ pointing from $\nu_1$ to $\nu_2$. Therefore, from \eqref{eq:gradwt1} and \eqref{eq:gradwt},  we infer
\[
\begin{split}
\mathbf{D^5}_{\Sg, e}(\nabla \, w) &= 
\frac{1}{|e|} \int_e \nabla \, w \cdot \tf_e \,{\rm d}s =
\mathbf{D}_{\VV, \nu_2}(w) - \mathbf{D}_{\VV, \nu_1}(w) =
\mathbf{D}_{\VV, \nu_2}(\widetilde{w}) - \mathbf{D}_{\VV, \nu_1}(\widetilde{w})
\\
&=
 \frac{1}{|e|} \int_e \nabla \, \widetilde{w} \cdot \tf_e \,{\rm d}s 
= \frac{1}{|e|} \int_e \ffi \cdot \tf_e \,{\rm d}s 
= \mathbf{D^5}_{\Sg, e}(\ffi) \,,
\end{split}
\]
that concludes the proof. 
\end{proof}

\begin{proposition}
\label{prp:inclusion2}
Let $\Sg_h$ and $\ZZ_h$ be the spaces defined in \eqref{eq:Sg_h} and \eqref{eq:Z_h} respectively. Then
\[
\CC \, \Sg_h = \ZZ_h \,.
\]
\end{proposition}
\begin{proof}
The proof follows by showing the following points:
\begin{itemize}
\item [${\boldsymbol {(i1)}}$]: for every $\ffi \in \Sg_h$, $\CC \, \ffi \in \ZZ_h$,
\item [${\boldsymbol {(i2)}}$]: for every $\vv \in \ZZ_h$ there exists $\ffi \in \Sg_h$  such that $\CC \, \ffi = \vv$.
\end{itemize}
Let us analyse the  inclusion ${\boldsymbol {(i1)}}$. 
Let $\ffi \in \Sg_h$, clearly
$\CC \, \ffi \in \ZZ(\Omega)$.
Therefore we need to verify that $(\CC \, \ffi)_{|P} \in \VV_h(P)$ for any $P \in \Omega_h$. 
It is evident that the second and the last line in definition \eqref{eq:Sg_h^P} 
correspond respectively to the $\CC$ version of the first and last line in definition \eqref{eq:V_h^P}.
Hence it remains to show that  $\CC \, \ffi$ is the velocity solution of the Stokes problem associated with definition \eqref{eq:V_h^P} on each element $P$.
A careful inspection of the biharmonic problem in definition \eqref{eq:Sg_h^P}, imply that the following are equivalent
\[
\begin{aligned}
&\int_P \dl \, \ffi \cdot \dl \, \ppsi \, {\rm d}P 
= \int_P \widetilde{\pp}_{k-1}   \cdot  \ppsi \, {\rm d}P 
\quad &\text{$\forall\ppsi \in \boldsymbol{\Psi}_0(P)$,}
\quad &\text{(by definition \eqref{eq:Sg_h^P})}
\\
&\int_P \dl \, \ffi \cdot \dl \, \ppsi \, {\rm d}P 
= \int_P \CC (\xx \wedge \pp_{k-1})   \cdot  \ppsi \, {\rm d}P 
\quad &\text{$\forall \ppsi\in \boldsymbol{\Psi}_0(P)$,}
\quad &\text{(characterization \eqref{eq:z_CC})}
\\
&\int_P \dl \, \ffi \cdot \dl \, \ppsi \, {\rm d}P 
= \int_P  (\xx \wedge \pp_{k-1})   \cdot \CC \, \ppsi \, {\rm d}P 
\quad &\text{$\forall \ppsi \in \boldsymbol{\Psi}_0(P)$.}
\quad &\text{(int. by parts + b.c)}
\end{aligned}
\]
In particular the last equation is still valid considering all $\ppsi \in \boldsymbol{\Psi}_0(P)\cap \ZZ(P)$. 
Therefore, using the identity $\dl = -\CC \, \CC + \Gr \, \dd$  and an integration by parts (coupled with the homogeneous boundary condition $\CC \, \ppsi = 0$ on $\partial P$), 
it can be proved that the following are equivalent 
\begin{equation}
\label{eq:bih-sto}
\begin{aligned}
&\int_P \dl \, \ffi \cdot \dl \, \ppsi \, {\rm d}P 
= \int_P  (\xx \wedge \pp_{k-1})   \cdot \CC \, \ppsi \, {\rm d}P 
\quad &\text{$\forall \ppsi \in \boldsymbol{\Psi}_0(P)\cap \ZZ(P)$,}
\quad &
\\
&\int_P \dl \, \ffi \cdot (- \CC \, \CC \, \ppsi) \, {\rm d}P 
= \int_P  (\xx \wedge \pp_{k-1})    \cdot \CC \, \ppsi \, {\rm d}P 
\quad &\text{$\forall \ppsi \in \boldsymbol{\Psi}_0(P) \cap \ZZ(P)$,}
\quad &
\\
&\int_P -\dl (\CC \, \ffi) \cdot \CC \, \ppsi \, {\rm d}P 
= \int_P  (\xx \wedge \pp_{k-1})   \cdot \CC \, \ppsi \, {\rm d}P 
\quad &\text{$\forall \ppsi \in \boldsymbol{\Psi}_0(P) \cap \ZZ(P)$.}
\quad &
\end{aligned}
\end{equation}
Exploiting Lemma 5.1 in \cite{girault-raviart:book}, for every $\zz \in \ZZ_0(P)$ there exists $\ppsi \in \boldsymbol{\Psi}_0(P) \cap \ZZ(P)$ such that $\zz = \CC \, \ppsi$.
Therefore the last equation in \eqref{eq:bih-sto} is equivalent to
\[
\int_P \Gr (\CC \, \ffi) : \Gr \zz \, {\rm d}P 
= \int_P  (\xx \wedge \pp_{k-1})   \cdot \zz \, {\rm d}P 
\quad \text{for all $\zz \in \ZZ_0(P)$,}
\]
and therefore $\vv = \CC \, \ffi$ is the velocity solution of a Stokes problem as in Definition \eqref{eq:V_h^P}.
That concludes the proof for $\boldsymbol{(i1)}$.

We focus now on $\boldsymbol{(i2)}$. 
Let $\vv \in \ZZ_h \subseteq \ZZ$, then from Corollary 3.3 in \cite{girault-raviart:book}
there exists $\widetilde{\ffi} \in [H^2(\Omega)]^3 \cap \ZZ(\Omega)$, such that $\CC \, \widetilde{\ffi} = \vv$.
Notice that, being $\widetilde{\ffi} \in [H^2(\Omega)]^3$ and $\CC \, \widetilde{\ffi}  \in  \VV_h(P)$ for any $P$ in $\Omega_h$,
it makes sense to compute
$\mathbf{D}_{\Sg}(\widetilde{\ffi})$.

Let us consider  the interpolant $\ffi \in \Sg_h$ of $\widetilde{\ffi}$ in the sense of DoFs, i.e. the function uniquely determined by (cf. Proposition \ref{prp:Sg_h})
\begin{equation}
\label{eq:dofsg}
\mathbf{D}_{\Sg}(\ffi) = \mathbf{D}_{\Sg}(\widetilde{\ffi}) \,.
\end{equation}
Property $\boldsymbol{(i1)}$ ensures $\CC \, \ffi \in \ZZ_h$. 
Therefore, employing Proposition \ref{prp:V_h^E_dofs}, $\ffi$ realizes $\boldsymbol{(i2)}$ if and only if $\mathbf{D}_{\VV}(\CC \,\ffi) = \mathbf{D}_{\VV}(\vv)$.
Is it straightforward to check that
\[
\mathbf{D^5}_{\VV}(\CC \,\ffi) = \mathbf{D^5}_{\VV}(\vv) = {\bf 0} \,,
\]
\[
\mathbf{D^{\boldsymbol i}}_{\VV}(\CC \,\ffi) =
\mathbf{D^{\boldsymbol i}}_{\Sg}(\ffi) =
\mathbf{D^{\boldsymbol i}}_{\Sg}(\widetilde{\ffi}) =
\mathbf{D^{\boldsymbol i}}_{\VV}(\vv)
\quad \text{for $\boldsymbol{i=1,2,3, 4}$,}
\]
except the face moment (that is slightly more subtle)
\[
\int_f \CC \,\ffi \cdot \nn_P^f \, {\rm d}f 
\quad \text{and} \quad
\int_f \vv \cdot \nn_P^f \, {\rm d}f 
\qquad \text{for any face $f$.}
\]
In order to show that the two quantities above are equal we exploit  the same computations in \eqref{eq:unisolvency2a} and \eqref{eq:dofsg} 
\[
\begin{split}
\int_f \CC \, \ffi \cdot \nn_P^f \,{\rm d}f &=
\sum_{e \in \partial f} |e| \, \mathbf{D^5}_{\Sg, e}(\ffi) \, \tf_e \cdot \tf_f^e 
=
\sum_{e \in \partial f} |e| \, \mathbf{D^5}_{\Sg, e}(\widetilde{\ffi}) \, \tf_e \cdot \tf_f^e 
\\
&=
\int_f \CC \,\widetilde{\ffi} \cdot \nn_P^f \,{\rm d}f 
= \int_f \vv \cdot \nn_P^f \,{\rm d}f \,.
\end{split}
\]
This ends the proof.
\end{proof}

\begin{proposition}
\label{prp:inclusion3}
Let $\VV_h$ and $Q_h$ be the spaces defined in \eqref{eq:V_h} and \eqref{eq:Q_h} respectively. Then 
\[
\dd \, \VV_h =  Q_h \,.
\]
\end{proposition}

\begin{proof}
We follow same strategy adopted in the previous propositions and show that
\begin{itemize}
\item [${\boldsymbol {(i1)}}$]: for every $\vv \in \VV_h$, $\dd \, \vv \in Q_h$,
\item [${\boldsymbol {(i2)}}$]: for every $q \in Q_h$ there exists $\vv \in \VV_h$  such that $\dd \, \vv = q$.
\end{itemize}
The inclusion ${\boldsymbol {(i1)}}$ is  trivial.
Regarding the point ${\boldsymbol {(i2)}}$, since \eqref{eq:exact} is an exact sequence, 
for any $q \in Q_h$ there exists $\widetilde{\vv} \in [H^1(\Omega)]^3$ such that $\dd \,\widetilde{\vv} = q$.
Now let $\vv \in \VV_h$  the function  uniquely determined by (cf. Proposition \ref{prp:V_h^E_dofs})
\begin{equation}
\label{eq:dofsv}
\begin{gathered}
\mathbf{D^1}_{\VV}(\vv) = \mathbf{D^2}_{\VV}(\vv) =  \mathbf{D^4}_{\VV}(\vv) = {\bf 0} \,,
\\
\mathbf{D^3}_{\VV}(\vv) = {\bf 0} \quad \text{except the face moments}
\quad
\int_f \vv \cdot \nn_P^f \, {\rm d}f  =
\int_f \widetilde{\vv} \cdot \nn_P^f \, {\rm d}f \,,
\\
\mathbf{D^5}_{\VV}(\vv) = \mathbf{D^5}_{\VV}(\widetilde{\vv}) \,.
\end{gathered}
\end{equation}
Notice that being $\widetilde{\vv} \in [H^1(\Omega)]^3$ the face moments in \eqref{eq:dofsv} and $\mathbf{D^5}_{\VV}(\widetilde{\vv})$ are actually well defined.
Therefore for any $P \in \Omega_h$ we infer
\begin{equation}
\label{eq:div1}
\int_P (\dd \, \vv) \, \widehat{\pp}_{k-1} \, {\rm d}P =
  \int_P (\dd \, \widetilde{\vv}) \, \widehat{\pp}_{k-1} \, {\rm d}P =
  \int_P q \, \widehat{\pp}_{k-1} \, {\rm d}P 
\qquad
\text{for all $\widehat{\pp}_{k-1} \in \widehat{\Pk}_{k-1 \setminus 0}(P)$.}
\end{equation}
Moreover employing the divergence theorem, \eqref{eq:dofsv} implies
\begin{equation}
\label{eq:div2}
\begin{split}
\int_P \dd \, \vv \, {\rm d}P &=
  \int_{\partial P} \vv \cdot \nn_P \, {\rm d}f =
  \sum_{f \in \partial P} \int_f \vv \cdot \nn_P^f \, {\rm d}f 
  \\
  &=
  \sum_{f \in \partial P} \int_f \widetilde{\vv} \cdot \nn_P^f \, {\rm d}f =
  \int_{\partial P} \widetilde{\vv} \cdot \nn_P \, {\rm d}f = 
  \int_P \dd \, \widetilde{\vv} \, {\rm d}P = 
  \int_P q \, {\rm d}P \,.
\end{split}   
\end{equation}
Notice that \eqref{eq:div1} and \eqref{eq:div2} coincide with 
$\mathbf{D}_Q(\dd \, \vv) = \mathbf{D}_Q(q)$ that coupled with $\dd \, \vv \in Q_h$  (from $\boldsymbol{(i1)}$) concludes the proof.
\end{proof}


\section{Virtual Elements for the 3-d Navier--Stokes equation}
\label{sec:VEM-ns}

We consider the steady Navier--Stokes equation on a polyhedral domain $\Omega \subseteq \R^3$ with homogeneous Dirichlet boundary
conditions:

\begin{equation}
\label{eq:ns continuous}
\left\{
\begin{aligned}
& \text{find $(\uu, \, p) \in [H^1_0(\Omega)]^3 \times L^2_0(\Omega)$, such that} \\
& \nu \, a(\uu, \,\vv) + c(\uu; \,  \uu, \vv) + b(\vv, p) = (\ff, \, \vv) \qquad & \text{for all $\vv \in [H^1_0(\Omega)]^3$,} \\
&  b(\uu, \, q) = 0 \qquad & \text{for all $q \in Q$,}
\end{aligned}
\right.
\end{equation}
where $\nu > 0$ represents the viscosity, $\ff \in [L^2(\Omega)]^3$ is the external force and  
\begin{align}
\label{eq:a-cont}
a(\uu,\, \vv) &:= \int_{\Omega} \epseps( \uu) : \epseps( \vv) \, {\rm d}\Omega 
\quad
&\text{for all $\uu$, $\vv \in [H^1(\Omega)]^3$,}
\\
\label{eq:c-cont}
c(\ww; \, \uu,\, \vv) &:= \int_{\Omega} [(\Gr \, \uu) \, \ww] \cdot  \vv \, {\rm d}\Omega 
\quad
&\text{for all $\ww$, $\uu$, $\vv \in [H^1(\Omega)]^3$,}
\\
\label{eq:b-cont}
b(\uu,\, q) &:= \int_{\Omega} \dd \, \uu  \, q \, {\rm d}\Omega 
\quad
&\text{for all $\uu \in [H^1(\Omega)]^3$ and $q \in L^2(\Omega)$.}
\end{align}
For sake of simplicity we here consider
Dirichlet homogeneous boundary conditions, different boundary conditions can be treated as well.

It is well known \cite{quarteroni-valli:book} that in the diffusion dominated regime
\[
\mathbf{(H)} \qquad \gamma := \frac{\|\ff\|_{-1}}{\nu^2} \ll 1
\]
the Navier--Stokes equation \eqref{eq:ns continuous} has a unique solution $(\uu, \, p)$ with
\[
|\uu|_1 \leq  \frac{\|\ff\|_{-1}}{\nu} \,.
\]
Moreover Problem \eqref{eq:ns continuous} can be formulated in the equivalent kernel form:
\begin{equation*}
\left\{
\begin{aligned}
& \text{find $\uu \in \ZZ_0(\Omega)$, such that} \\
& \nu \, a(\uu, \,\vv) + c(\uu; \,  \uu, \vv)  = (\ff, \, \vv) \qquad & \text{for all $\vv \in \ZZ_0(\Omega)$.} 
\end{aligned}
\right.
\end{equation*}

\subsection{Discrete forms and load term approximation}
\label{sub:forms}

In this subsection we briefly describe the
construction of a discrete version of the bilinear form
$a(\cdot, \cdot)$  given in \eqref{eq:a-cont}
and trilinear form $c(\cdot; \cdot, \cdot)$ given in \eqref{eq:c-cont}.
We can follow in a rather slavish way the procedure initially introduced in \cite{volley} for the laplace problem and further developed in \cite{BLV:2018} for flow problems.
First, we decompose into local contributions the bilinear
form $a(\cdot, \cdot)$ and the trilinear form $c(\cdot; \cdot, \cdot)$ by considering 
\begin{equation*}
a(\uu, \, \vv) =: \sum_{P \in \Omega_h} a^P(\uu, \, \vv) \,,
\qquad
c(\ww; \, \uu, \vv) =: \sum_{P \in \Omega_h} c^P(\ww; \, \uu, \vv) \,,
\end{equation*}
for all $\ww$, $\uu$, $\vv \in [H^1(\Omega)]^3$.

As usual in VEM framework the discrete counterpart of the continuous forms above are defined starting from the polynomial projections defined in \eqref{eq:P0_k^E} and \eqref{eq:Pn_k^E}.
The following proposition extends to the 3-d case the  result for the bi-dimensional spaces \cite{BLV:2017, vacca:2018}.

\begin{proposition}
\label{prp:proj}
Let $\widehat{\B}_k(f)$ and $\VV_h(P)$ be the spaces defined in \eqref{eq:Bhat_h^n} and \eqref{eq:V_h^P} respectively. 
The DoFs $\mathbf{D}_{\VV}$ allow us to compute exactly the face projections
\begin{equation*}
\Pi_{k}^{\nabla, f} \colon [\widehat{\B}_k(f)]^3 \to [\Pk_k(f)]^3 \,,
\qquad
\Pi_{k+1}^{0, f}  \colon [\widehat{\B}_k(f)]^3 \to [\Pk_{k+1}(f)]^3 \,,
\end{equation*}
for any $f \in \partial P$, and the element projections
\begin{equation*}
\begin{aligned}
\Pi_{k}^{\nabla, P} &\colon \VV_h(P) \to [\Pk_k(P)]^3 \,, 
\\
{\boldsymbol \Pi}_{k-1}^{0, P} &\colon \Gr ( \VV_h(P) ) \to [\Pk_{k-1}(P)]^{3 \times 3} \,,
\\
\Pi_{k}^{0, P} &\colon \VV_h(P) \to [\Pk_k(P)]^3 \,,
\end{aligned}
\end{equation*}
in the sense that, given any $\vv_h \in \VV_h(P)$, 
we are able to compute the polynomials 
\[
\Pi_{k}^{\nabla, f} \vv_h \,, \qquad
\Pi_{k+1}^{0, f} \vv_h \,, \qquad
\Pi_{k}^{\nabla, P} \vv_h \,, \qquad
{\boldsymbol \Pi}_{k}^{0, P} (\Gr \, \vv_h) \,, \qquad
\Pi_{k}^{0, P} \vv_h \,, 
\]
using only, as unique information, the DoFs values $\mathbf{D}_{\VV}$ of $\vv_h$.

\end{proposition}

\begin{proof}
The computability of the face projections is a direct application of Remark 5 in \cite{projectors}.
Concerning the element projections we here limit to prove the last item, the first two follow analogous techniques.

By definition of $L^2$-projection \eqref{eq:P0_k^E},
in order the determine, for any $\vv \in \VV_h(P)$, the polynomial $\Pi_k^{0, P} \vv$ 
we need to compute
\[
\int_P \vv \cdot \pp_k \, {\rm d}P   
\qquad
\text{for all $\pp_k \in [\Pk_k(P)]^3$.}
\]
From polynomial decomposition \eqref{eq:poly_decomposition_3D_grad} we can write
\[
\pp_k = \nabla \, \widehat{p}_{k+1}  + \xx \wedge \widehat{\boldsymbol q}_{k-1}
+  \xx \wedge{\boldsymbol q}_{k-3} 
\]
for some $\widehat{p}_{k+1} \in \widehat{\Pk}_{k+1 \setminus 0}(P)$,
$\widehat{\boldsymbol q}_{k-1} \in [\widehat{\Pk}_{k-1 \setminus k-3}(P)]^3$,
${\boldsymbol q}_{k-3} \in [{\Pk}_{k-3}(P)]^3$.
Thus
\[
\begin{aligned}
\int_P \vv \cdot \pp_k \, {\rm d}P = &
\int_P \vv \cdot \left(\nabla \, \widehat{p}_{k+1}  + \xx \wedge \widehat{\boldsymbol q}_{k-1}
+ \xx \wedge {\boldsymbol q}_{k-3} \right) \,{\rm d}P
\\ 
= &
\int_P \Pi_k^{\nabla, P}\vv \cdot (\xx \wedge \widehat{\boldsymbol q}_{k-1} ) \,{\rm d}P
+ \int_P \vv \cdot (\xx \wedge {\boldsymbol q}_{k-3} ) \,{\rm d}P \,+
\quad & \text{(enhancing def. \eqref{eq:V_h^P})}
\\
& - \int_P (\dd \, \vv) \, \widehat{p}_{k+1} \,{\rm d}P
+ \int_{\partial P}  \vv \cdot \nn_P \, \widehat{p}_{k+1} \, {\rm d}f
\quad & \text{(integration by parts)}
\\ 
= &
\int_P \Pi_k^{\nabla, P}\vv \cdot (\xx \wedge \widehat{\boldsymbol q}_{k-1} ) \,{\rm d}P
+ \int_P \vv \cdot (\xx \wedge {\boldsymbol q}_{k-3} ) \,{\rm d}P \,+
\\
& - \int_P (\dd \, \vv) \, \widehat{p}_{k+1} \,{\rm d}P
+ \sum_{f \in \partial P}
\int_{f}  (\Pi_{k+1}^{0, f} \vv) \cdot \nn_P^f \, \widehat{p}_{k+1} \, {\rm d}f
\quad & \text{(by def. \eqref{eq:P0_k^E})}
\end{aligned}
\]
The first and the last integrals are computable being $\Pi_k^{\nabla, P}\vv$ and $\Pi_{k+1}^{0, f}$ computable.
The second addend corresponds to the DoFs $\mathbf{D^4}_{\VV}$.
For the third addend we observe that, since $\dd \, \vv$ is a polynomial of degree less or equal than $k-1$ we can exactly reconstruct its value
from the DoFs $\mathbf{D^5}_{\VV}$ and the normal face moments in $\mathbf{D^3}_{\VV}$.
\end{proof}

On the basis of the projections above, following a standard procedure in the VEM framework,
we define the computable (in the sense of Proposition \ref{prp:proj}) discrete local forms and the approximated right hand side
\begin{align}
\label{eq:ahP}
a_h^P(\uu, \, \vv) &:= 
\int_P 
\left({\boldsymbol \Pi}_{k-1}^{0, P} \epseps  (\uu ) \right) :
\left({\boldsymbol \Pi}_{k-1}^{0, P} \epseps  (\vv ) \right) \, {\rm d}P
+ \mathcal{S}^P \left( (I - \Pi_k^{\gr, P} ) \uu, \, (I - \Pi_k^{\gr, P} ) \vv \right) \,,
\\
\label{eq:chP}
c_h^P(\ww; \, \uu, \vv) &:= 
\int_P \left[ \left(\boldsymbol{\Pi}_{k-1}^{0, P} \Gr \, \uu \right) \,  \Pi_k^{0, P} \ww\right] \cdot \Pi_k^{0, P} \vv \,{\rm d}P \,,
\\
\label{eq:fhP}
(\ff_h, \, \vv)_P &:= \int_P  \Pi_k^{0, P} \ff \cdot \vv \, {\rm d}P \,,
\end{align}
for all $\ww$, $\uu$, $\vv \in \VV_h(P)$, where clearly
\[
{\boldsymbol \Pi}_{k-1}^{0, P} \epseps  (\vv) = 
\frac{{\boldsymbol \Pi}_{k-1}^{0, P} \Gr \,  \vv + 
({\boldsymbol \Pi}_{k-1}^{0, P} \Gr \,  \vv)^{\rm  T}
}{2}
\]
and the symmetric stabilizing  form 
 $\mathcal{S}^P \colon \VV_h(P) \times \VV_h(P) \to \R$ satisfies
\[
|\vv|_{1, P}^2 \lesssim \mathcal{S}^P(\vv, \, \vv) \lesssim |\vv|_{1, P}^2
\qquad \text{for all $\vv \in  {\rm Ker}(\Pi_k^{\gr, P})$.}
\]
The condition above essentially requires the stabilizing term  $\mathcal{S}^P(\vv, \, \vv)$  to scale as  $|\vv|_{1, P}^2$. 
For instance, a standard choice for the stabilization is the  $D$-recipe stabilization introduced in \cite{BDR:2017}.

\begin{remark}
\label{rm:projectionPDbis}
The $H^1$-seminorm projection $\Pi^{\nabla, P}_k$ in the stabilization term of Definition  \eqref{eq:ahP} can be replaced by any polynomial projection $\Pi^{P}_k$ that is computable on the basis of the DoFs $\mathbf{D}_{\VV}$ (in the sense of Proposition \ref{prp:proj}). 
A possible choice will be explored in Section \ref{sec:num_test}.
\end{remark}

The global virtual forms and the global approximated right-hand side
are defined by simply
summing the local contributions:
\begin{equation}
\label{eq:formh}
a_h(\uu, \, \vv) := \sum_{P \in \Omega_h} a_h^P(\uu, \,\vv)\,,
\quad
c_h(\ww; \, \uu, \, \vv) := \sum_{P \in \Omega_h} c_h^P(\ww; \, \uu, \, \vv)\,,
\quad
(\ff_h, \, \vv) := \sum_{P \in \Omega_h} (\ff_h, \,\vv)_P \,,
\end{equation}
for all $\ww$, $\uu$, $\vv \in \VV_h$.

%
%

\subsection{ The discrete problem}
\label{seb:discrete-problme}

Referring to the discrete spaces \eqref{eq:V_h}, \eqref{eq:Q_h}, the discrete forms and the approximated load term \eqref{eq:formh} and the $\dd$ form \eqref{eq:b-cont},
the virtual element approximation of the
Navier--Stokes equation  is given by
\begin{equation}
\label{eq:ns virtual}
\left\{
\begin{aligned}
& \text{find $(\uu_h, \,  p_h) \in \VV_{h, 0} \times Q_{h, 0}$, such that} \\
& \nu \, a_h(\uu_h, \, \vv_h) + c_h(\uu_h; \,  \uu_h, \vv_h) + b(\vv_h, \, p_h) = (\ff_h, \, \vv_h) \qquad & \text{for all $\vv_h \in \VV_{h, 0}$,} \\
&  b(\uu_h, \, q_h) = 0 \qquad & \text{for all $q_h \in Q_{h, 0}$,}
\end{aligned}
\right.
\end{equation}
where $\VV_{h, 0} := \VV_h \cap [H^1_0(\Omega)]^3$ and $Q_{h, 0} := Q_h \cap L^2_0(\Omega)$.

Recalling the kernel inclusion \eqref{eq:kernel_inclusion},
Problem~\eqref{eq:ns virtual} can be also formulated in the equivalent kernel form
\begin{equation}
\label{eq:nsvirtual ker}
\left\{
\begin{aligned}
& \text{find $\uu_h \in \ZZ_{h, 0}$, such that} \\
& \nu \, a_h(\uu_h, \, \vv_h) + c_h(\uu_h; \, \uu_h, \vv_h) = (\ff_h, \vv_h) \qquad & \text{for all $\vv_h \in \ZZ_{h, 0}$,} 
\end{aligned}
\right.
\end{equation}
with the obvious notation $\ZZ_{h, 0} := \ZZ_h \cap [H^1_0(\Omega)]^3$. 

%

Combining the arguments in \cite{BLV:2017, BLV:2018, brenner-sung:2018} it is possible to show that the virtual space $\VV_h$ has  an optimal interpolation order of accuracy with respect to the degree $k$,
and that the couple of spaces $(\VV_h, \, Q_h)$ is inf-sup stable \cite{boffi-brezzi-fortin:book}.
The following existence and convergence theorem extends the analogous result for the bi-dimensional case \cite{BLV:2018}.

\begin{theorem}
\label{thm:u}
Under the assumptions  $\mathbf{(A1)}$, $\mathbf{(A2)}$, $\mathbf{(A3)}$ and
 and $\mathbf{(H)}$, 
let $(\uu, \, p) \in [H^1_0(\Omega)]^3 \times L^2_0(\Omega)$ be the solution of Problem \eqref{eq:ns continuous} and $(\uu_h, \, p_h) \in \VV_{h, 0} \times Q_{h, 0}$ be the (unique) solution of Problem \eqref{eq:ns virtual}. 
Assuming moreover $\uu, \ff \in [H^{s+1}(\Omega)]^3$ and $p \in H^s(\Omega)$, $0 < s \leq k$, then 
\begin{gather}
\label{eq:thm:u}
| \uu - \uu_h |_{1} \lesssim \, h^{s} \, \mathcal{F}(\uu; \, \nu, \gamma) + \, h^{s+2} \, \mathcal{H}(\ff; \nu) \,,
\\
\label{eq:p-est}
\|p  - p_h\|_0 \lesssim \, h^{s} \, |p|_{s} +  h^s \, \mathcal{K}(\uu; \nu, \gamma) +  h^{s+2} \, |\ff|_{s+1} 
\end{gather}
for suitable functions 
$\mathcal{F}$, 
$\mathcal{H}$,
$\mathcal{K}$ 
independent of $h$. 
\end{theorem}

Note that, as a consequence of the important property \eqref{eq:kernel_inclusion}, there is no direct dependence of the velocity error on the pressure solution.

\begin{remark}
\label{rm:curlformulation}
Since Proposition \ref{prp:inclusion2} yields an explicit characterization of $\ZZ_h$ as $\CC \, \Sg_h$, one could follow \eqref{eq:nsvirtual ker} and build an equivalent $\CC$ (discrete) formulation (see for instance Problem (77) in \cite{BMV:2018}).
Such approach is less appealing in 3-d since the $\CC$ operator has a non trivial kernel and thus some stabilization or additional Lagrange multiplier would be needed in the formulation. 
Moreover this approach does not seem to be competitive in terms of number of DoFs with the reduced version of the method (see Subsection \ref{sub:reduced}).
As a consequence, we do not explore any scheme resulting from the $\CC$ formulation.  
\end{remark}

\subsection{Reduced spaces and reduced scheme}
\label{sub:reduced}

In the present section we briefly 
show that Problem \eqref{eq:ns virtual} is somehow equivalent to a suitable reduced
problem  entangling relevant fewer DoFs, especially for
large $k$. 
This reduction is analogous to its two-dimensional counterpart in Section 5 in \cite{BLV:2017} and Section 5.2 in \cite{vacca:2018}.

The core idea is that $\mathbf{D^5}_{\VV}(\uu_h) = {\bf 0}$, where $\uu_h$ denotes the solution of \eqref{eq:ns virtual}, and therefore such degrees of freedom (and also the associated pressures) can be trivially eliminated from the system.
Hence on each polygon $P$, let us define the reduced local  spaces:
\begin{equation*}
\begin{aligned}
\widetilde{\VV}_h(P) := \biggl\{  
\vv \in [H^1(P)]^3 \, \, \, \text{s.t.}  \, \, \,
& \vv_{|\partial P} \in [\widehat{\B}_k(\partial P)]^3
\\
&  \biggl\{
\begin{aligned}
& \boldsymbol{\Delta}    \vv  +  \nabla s \in \xx \wedge [\Pk_{k-1}(P)]^3,  \\
& {\rm div} \, \vv \in \Pk_{0}(P),
\end{aligned}
\biggr. \quad  \text{ for some $s \in  L^2_0(P)$} 
\\
\biggl.
&  \left(  \vv - \Pi_k^{\nabla, P} \vv  , \, \xx \wedge \widehat{\pp}_{k-1} \right)_P  = 0 \quad
\text{for all $\widehat{\pp}_{k-1} \in  [\widehat{\Pk}_{k-1 \setminus k-3}(P)]^3$}
\biggr \}
\end{aligned}
\end{equation*}
and
\begin{equation*}
\widetilde{Q}_h(P) := \Pk_{0}(P) \,.
\end{equation*}
Exploiting the same tools of Proposition \ref{prp:V_h^E_dofs} it can be proved that 
the  linear operators $\widetilde{\mathbf{D}}_{\VV}$, split into four subsets, defined by
\[
\mathbf{\widetilde{D}^{\boldsymbol{i}}}_{\VV} =
\mathbf{D^{\boldsymbol{i}}}_{{\VV}} 
\qquad \text{for $\boldsymbol{i = 1, 2, 3, 4}$,}
\]
constitute a set of DoFs for $\widetilde{\VV}_h(P)$.
Concerning the space $\widetilde{Q}_h(P)$, it is straightforward to see that 
$\dim(\widetilde{Q}_h(P)) = 1$ with unique DoF $\widetilde{\mathbf{D}}_{Q}$ defined by
$
\widetilde{\mathbf{D}}_{Q}(q) :=  \int_P q \,{\rm d}P
$.
The global spaces $\widetilde{\VV}_h$ and $\widetilde{Q}_h$ are obtained in the standard fashion by gluing the local spaces:
\begin{align}
\label{eq:Vr_h}
\widetilde{\VV}_h &:= \{\vv \in [H^1(\Omega)]^3 \quad \text{s.t.} \quad \vv_{|P} \in \widetilde{\VV}_h(P) \} \,,
\\
\label{eq:Qr_h}
\widetilde{Q}_h &:= \{q \in L^2(\Omega) \quad \text{s.t.} \quad q_{|P} \in \widetilde{Q}_h(P) \} \,.
\end{align}
%
We remark that by construction $\ZZ_h \subseteq \widetilde{\VV}_h$, therefore employing Proposition \ref{prp:inclusion1} and Proposition \ref{prp:inclusion2} we can state the following result.

\begin{proposition}
\label{prp:virtualr_exact}
Referring to \eqref{eq:W_h}, \eqref{eq:Sg_h}, \eqref{eq:Vr_h} and \eqref{eq:Qr_h}, the sequence
\[
\R \, \xrightarrow[]{ \,\quad \text{{$i$}} \quad \,   } \,
W_h \, \xrightarrow[]{   \quad \text{{$\gr$}} \quad  }\,
\Sg_h \, \xrightarrow[]{ \, \, \, \text{{$\CC$}} \, \, \, }\,
\widetilde{\VV}_h \, \xrightarrow[]{  \, \, \, \, \text{{$\dd$}} \, \, \, \, }\,
\widetilde{Q}_h \, \xrightarrow[]{ \quad 0 \quad }
0 
\]
is an exact sub-complex of \eqref{eq:exact}.
\end{proposition}

 Referring to 
\eqref{eq:Vr_h}, \eqref{eq:Qr_h} and \eqref{eq:formh}, we consider the reduced problem:
\begin{equation}
\label{eq:ns virtualr}
\left\{
\begin{aligned}
& \text{find $(\widetilde{\uu}_h, \,  \widetilde{p}_h) \in \widetilde{\VV}_{h, 0} \times \widetilde{Q}_{h, 0}$, such that} \\
& \nu \, a_h(\widetilde{\uu}_h, \, \vv_h) + c_h(\widetilde{\uu}_h; \,  \widetilde{\uu}_h, \vv_h) + b(\vv_h, \, \widetilde{p}_h) = (\ff_h, \, \vv_h) \qquad & \text{for all $\vv_h \in \widetilde{\VV}_{h, 0}$,} \\
&  b(\widetilde{\uu}_h, \, q_h) = 0 \qquad & \text{for all $q_h \in \widetilde{Q}_{h, 0}$,}
\end{aligned}
\right.
\end{equation}
where $\widetilde{\VV}_{h, 0} := \widetilde{\VV}_h \cap [H^1_0(\Omega)]^3$ and $\widetilde{Q}_{h, 0} := \widetilde{Q}_h \cap L^2_0(\Omega)$.

It is trivial to check that the reduced scheme \eqref{eq:ns virtualr} has
$ (2 \, \pi_{k-1, 3} - 2) \, N_P$ degrees of freedom less when compared with the original one \eqref{eq:ns virtual}.

The following proposition is easy to check and states the relation between Problem \eqref{eq:ns virtual} and the reduced Problem \eqref{eq:ns virtualr}.

\begin{proposition}
\label{prp:equivalent}
Let $(\uu_h, \, p_h) \in \VV_h \times Q_h$  and   $(\widetilde{\uu}_h, \, \widetilde{p}_h) \in \widetilde{\VV}_h \times \widetilde{Q}_h$ be the solution of Problem \eqref{eq:ns virtual} and Problem \eqref{eq:ns virtualr} respectively. Then 
\[
\widetilde{\uu}_h = \uu_h   \qquad \text{and} \qquad
\widetilde{p}_h = \Pi_0^{0, P} p_h \quad \text{in every $P \in \Omega_h$.}
\]
\end{proposition}

\newcommand{\Franco}[1]{\noindent{\color{green}\textbf{[F:~#1]}}}
\newcommand{\X}{\mathbf{x}}

\section{Numerical validation}\label{sec:num_test}

In this section we numerically verify the proposed discretization scheme. 
Before dealing with such examples,
we briefly describe an alternative (computationally cheaper) projection adopted in the implementation of the method.
Then we outline the polyhedral meshes and the error norms 
used in the analysis.

\subsection{An alternative DoFs-based projection}
\label{sub:dofs-proj}

In the light of Remark \ref{rm:projectionPD} and Remark \ref{rm:projectionPDbis},
the aim of the present subsection is to exhibit an alternative projection to be used in the place of the standard $H^1$-seminorm projection $\Pi^{\nabla, P}_k$ in \eqref{eq:V_h^P} and \eqref{eq:ahP} that will turn out to be very easy to implement.
An analogous alternative projection could also be used to substitute $\Pi^{\nabla, f}_k$ in \eqref{eq:Bhat_h^n}.

For any element $P \in \Omega_h$, let $\texttt{NDoF} := \dim(\VV_h(P))$.
Then referring to Proposition \ref{prp:V_h^E_dofs} we set 
$\mathbf{D}_{\VV} := \{\mathbf{D}_{\VV, i}\}_{i=1}^{\texttt{NDoF}}$,
and we denote with $\mathcal{D} \colon\VV_h(P) \to \R^{\texttt{NDoF}}$  the linear operator defined for all
$\vv \in \VV_h(P)$ by
\[
\left(\mathcal{D} \, \vv \right)_i = \mathbf{D}_{\VV, i} (\vv)  \qquad \text{for $i=1, \dots, \texttt{NDoF}$,}
\]
i.e. $\mathcal{D} \, \vv$ is the vector containing the degree of freedom values $\mathbf{D}_{\VV}$ associated to $\vv$.
We consider:
\begin{itemize}
\item the \textbf{DoFs-projection} ${\Pi}_{n}^{\mathcal{D},P} \colon \VV_h(P) \to [\Pk_n(P)]^3$    defined for any $\vv \in \VV_h(P)$ by 
\begin{equation}
\label{eq:PD_k^E}
\left(\mathcal{D}\, {\boldsymbol q}_n \,, \mathcal{D}\, (\vv - {\Pi}_{n}^{\mathcal{D},P} \, \vv )   \right)_{\R^{\texttt{NDoF}}} = 0
\qquad  \text{for all ${\boldsymbol q}_n \in [\Pk_n(P)]^3$.} 
\end{equation} 
\end{itemize}
Notice that ${\Pi}_{n}^{\mathcal{D},P}$ is a special case of the serendipity projection introduced in
\cite{BBMR:2016:serendipity}.

Although the projection  ${\Pi}_{n}^{\mathcal{D},P}$ may seem awkward on paper, it is quite simple and cheap to implement on the computer (since it is nothing but an euclidean projection with respect to the degree of freedom vectors). Indeed, it can be checked that the matrix formulation $\boldsymbol{\Pi}_{n}^{\mathcal{D},P}$ of the operator ${\Pi}_{n}^{\mathcal{D},P}$ acting from $\VV_h(P)$ to $\VV_h(P)$ (containing $[\Pk_n(P)]^3$) with respect to the basis $\boldsymbol{\mathcal{V}}$
(cf. \cite{autostoppisti}, formula (3.18))
is 
\[
\boldsymbol{\Pi}_{n}^{\mathcal{D},P} = D \, (D^T D)^{-1} D^T \in \R^{\texttt{NDoF} \times \texttt{NDoF}}\,,
\]
where $D \in \R^{\texttt{NDoF} \times 3\pi_{n, 3}}$ is the matrix defined by (cf. \cite{autostoppisti}, formula (3.17))
\[
D_{i,\alpha} := \mathbf{D}_{\VV, i} ({\boldsymbol m}_{\alpha})
\quad \text{for $i=1, \dots, \texttt{NDoF}$ and 
$\alpha =1, \dots, 3\pi_{n, 3}$,}
\]
where using standard VEM notation, ${\boldsymbol m}_{\alpha}$ denotes the scaled monomial
\[
{\boldsymbol m}_{\alpha} := 
\left(
\left(\frac{\xx - \xx_B}{h_P} \right)^{\boldsymbol{\alpha_1}}, \quad
\left(\frac{\xx - \xx_B}{h_P} \right)^{\boldsymbol{\alpha_2}}, \quad  
\left(\frac{\xx - \xx_B}{h_P} \right)^{\boldsymbol{\alpha_3}}
\right)^T
\] 
with $\xx_B$ barycenter of the polyhedron $P$, and ${\boldsymbol{\alpha_1}}$, ${\boldsymbol{\alpha_2}}$ and ${\boldsymbol{\alpha_3}}$ suitable multi-indexes.

\subsection{Meshes and error norms}

We consider the standard $[0,\,1]^3$ cube as domain $\Omega$ 
and we make four different discretizations of such domain: 
\begin{enumerate}[a)]
 \item \textbf{Structured} refers to meshes composed by structured cubes inside the domain, Figure~\ref{fig:meshes}  (a).
 \item \textbf{Tetra} is a constrained Delaunay tetrahedralization of $\Omega$, Figure~\ref{fig:meshes} (b).
 \item \textbf{CVT} refers to a Centroidal Voronoi Tessellation, i.e., 
 a Voronoi tessellation where the control points coincide with the centroid of the cells they define, Figure~\ref{fig:meshes} (c).
 \item \textbf{Random} is a Voronoi diagram of a point set randomly displayed inside the domain $\Omega$, Figure~\ref{fig:meshes} (d).
\end{enumerate}

\begin{figure}[!htb]
\begin{center}
\begin{tabular}{cc}
\includegraphics[width=0.37\textwidth]{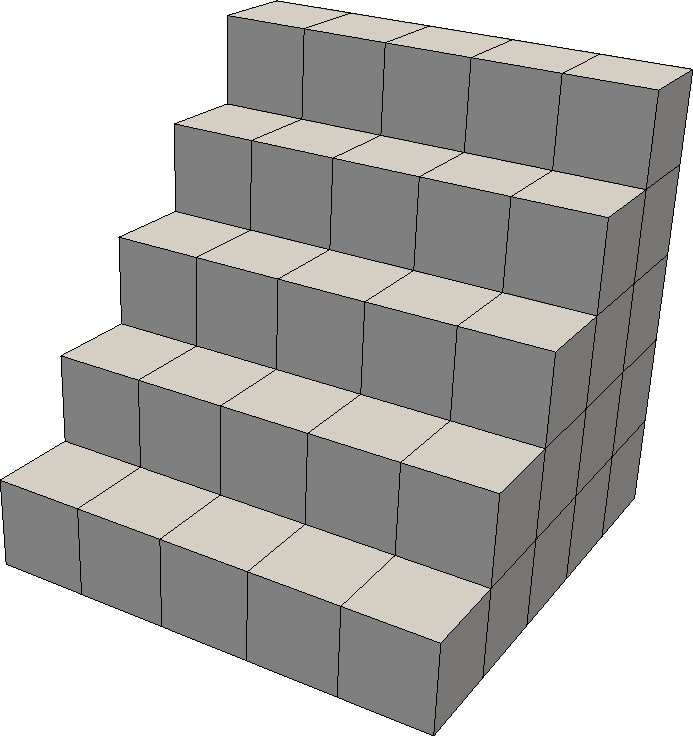} &     
\includegraphics[width=0.37\textwidth]{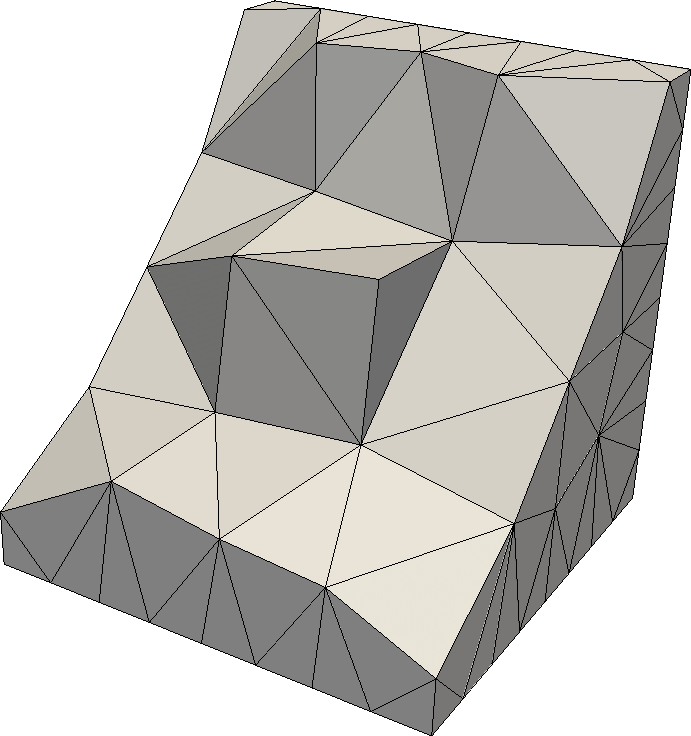} \\ 
(a) & (b) \\
\includegraphics[width=0.37\textwidth]{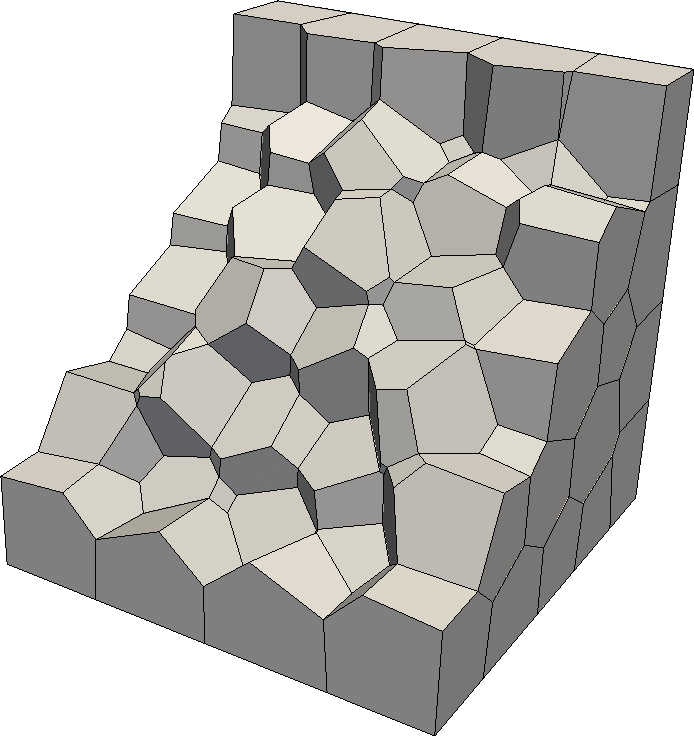} &           
\includegraphics[width=0.37\textwidth]{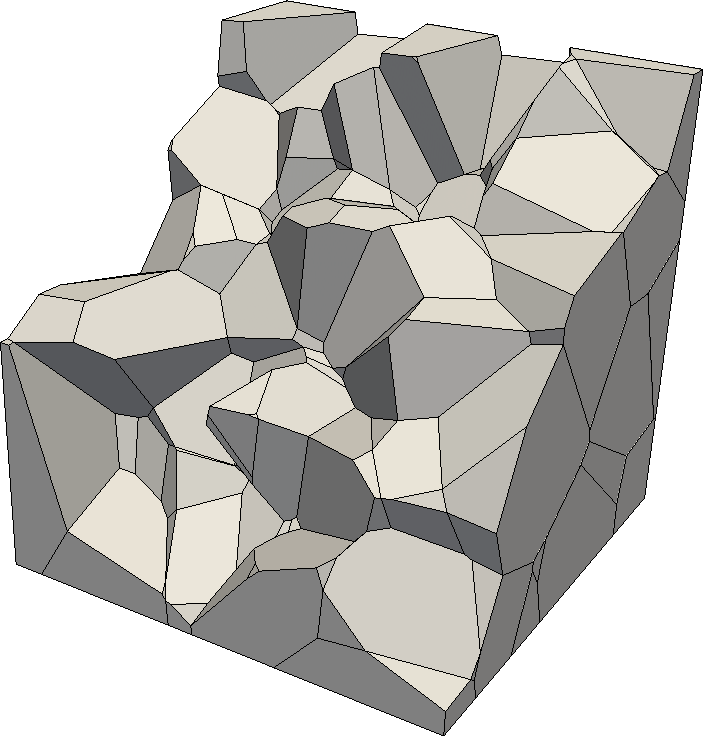} \\  
(c) & (d) \\ 
\end{tabular}
\end{center}
\caption{Adopted mesh types: (a) structured, (b) tetra, (c) CVT and (d) random.}
\label{fig:meshes}
\end{figure}

We would like to underline that the last type of mesh will severely test the robustness of the proposed method.
Indeed, \textbf{Random} meshes are characterized by elements whose faces can be very small and distorted, 
see the detail in Figure~\ref{fig:meshes} (d).

The tetrahedral meshes are generated via \texttt{tetgen}~\cite{si2015tetgen}, 
while the last two are obtained by exploiting the c++ library \texttt{voro++}~\cite{voroPlusPlus}.
To analyze the error convergence rate,
we make, for each family, a sequence of four meshes with decreasing size.
For each mesh we define the mesh-size as
$$
h := \frac{1}{L_P}\sum_{P\in\Omega_h} h_P\,.
$$

Let $(\uu,p)$ and $(\uu_h,p_h)$ be the continuous and discrete VEM solution of the Stokes (or Navier-Stokes problem) under study.
To evaluate how this discrete solution is close to the exact one,
we use the following error measures, that make use of the local projection described in Proposition~\ref{prp:proj}:
\begin{itemize}
\item \textbf{$H^1$--velocity error}:
$$
e_{H^1}^{\uu} := \sqrt{\sum_{P\in\Omega_h}\big|\big| \nabla\uu - {\boldsymbol \Pi}_{k-1}^{0, P}\,\nabla\uu_h \big|\big|^2_{L^2(P)}}\,,
$$
the theoretical expected convergence rate is $h^k$ (cf. \eqref{eq:thm:u});
\item \textbf{$L^2$--pressure error}:
$$
e_{L^2}^p := \sqrt{\sum_{P\in\Omega_h}\big|\big|p - p_h \big|\big|^2_{L^2(P)}}\,,
$$
the expected rate is $h^{k}$ (cf. \eqref{eq:p-est}).
\end{itemize}

\subsection{Numerical tests}

In this subsection we consider three different tests. 
In the first two examples, 
we numerically verify the theoretical trend of all the errors
for a Stokes and Navier--Stokes problem.
Finally, we propose two benchmark examples for the Stokes equation with the property of
having the velocity solution in the discrete space $\VV_h$.
It is well known that classical mixed finite element methods lead in this situations to significant velocity errors, stemming from the velocity/pressure coupling in the error estimates. This effect is greatly reduced by the
presented methods (cf. Theorem \ref{thm:u}, estimate \eqref{eq:thm:u}).

\paragraph{Example 1 (Stokes problem).}
In this section we solve the Stokes problem on the unit cube $[0\,,1]^3$,
the discreted version being as in~\eqref{eq:ns virtual} but 
without the trilinear form $c_h(\cdot\,;\cdot,\cdot)$.
We consider Neumann homogeneous boundary conditions on the faces associated with the planes $x=0$ and $x=1$.
The load term and the Dirichlet boundary conditions on the remaining faces are chosen in such a way that the exact solution is 
$$
\uu(x,\,y,\,z):=
\left(
\begin{array}{r}
\phantom{-2}\sin(\pi x)\,\cos(\pi y)\,\cos(\pi z)\\
\phantom{-2}\cos(\pi x)\,\sin(\pi y)\,\cos(\pi z)\\
{-2}\cos(\pi x)\,\cos(\pi y)\,\sin(\pi z)
\end{array}
\right)
$$
and
$$
p(x,\,y,\,z):= -\pi\,\cos(\pi x)\,\cos(\pi y)\,\cos(\pi z)\,.
$$

We consider the \textbf{Structured}, \textbf{CVT} and \textbf{Random} meshes.
In Figure~\ref{fig:ese1Conv} we show the behaviour of the errors $e_{H^1}^{\uu}$ and $e_{L^2}^p$.
The slope of such errors are the expected ones, $O(h^k)$ see Theorem~\ref{thm:u}.
Moreover, for each approximation degree $k$ the convergence lines associated with different meshes are close to each other and 
this represents a numerical evidence that the proposed method is robust with respect to the adopted meshes.

\begin{figure}[!htb]
\begin{center}
\begin{tabular}{cc}
\includegraphics[height=0.37\textwidth]{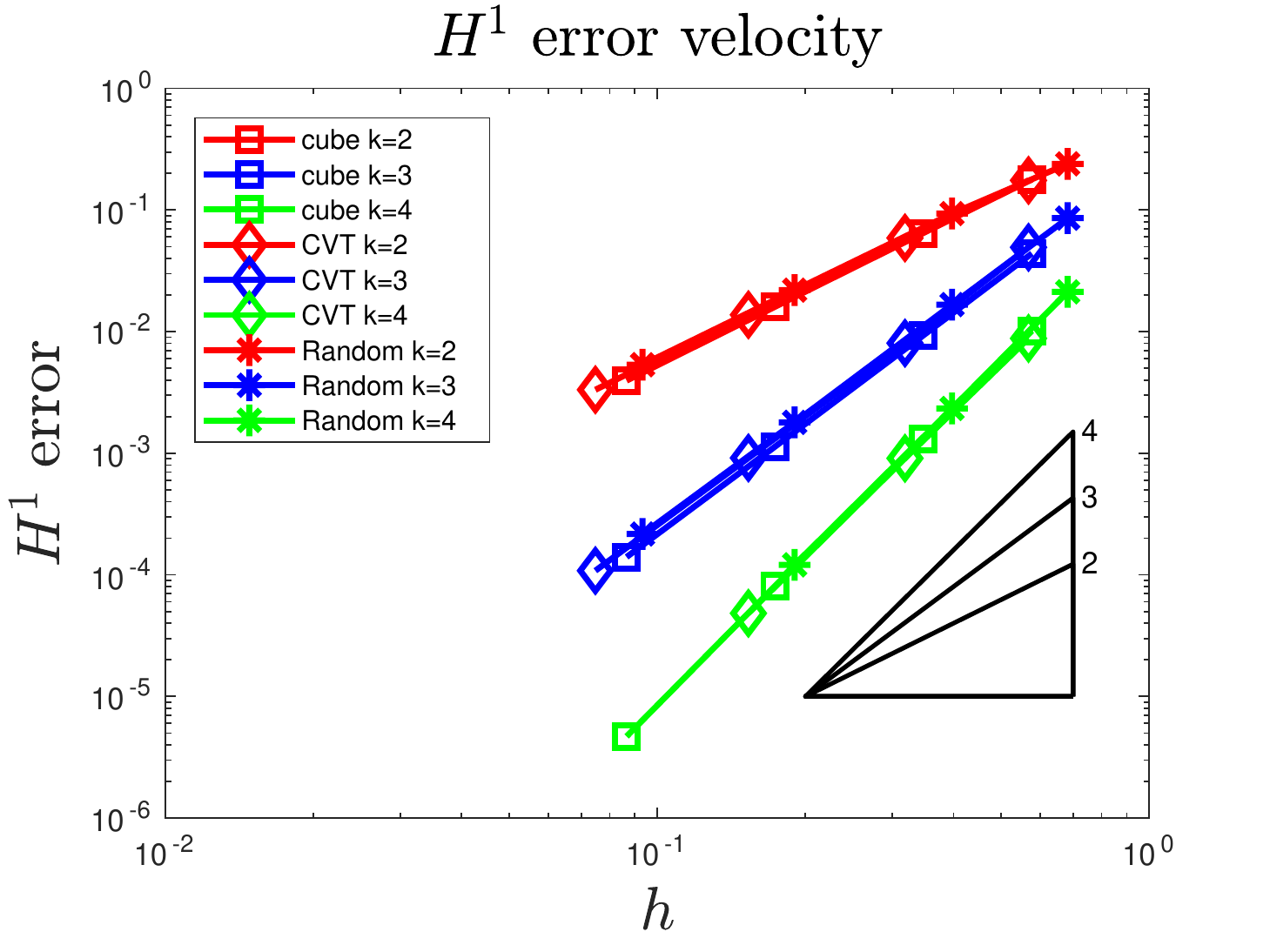} &
\includegraphics[height=0.37\textwidth]{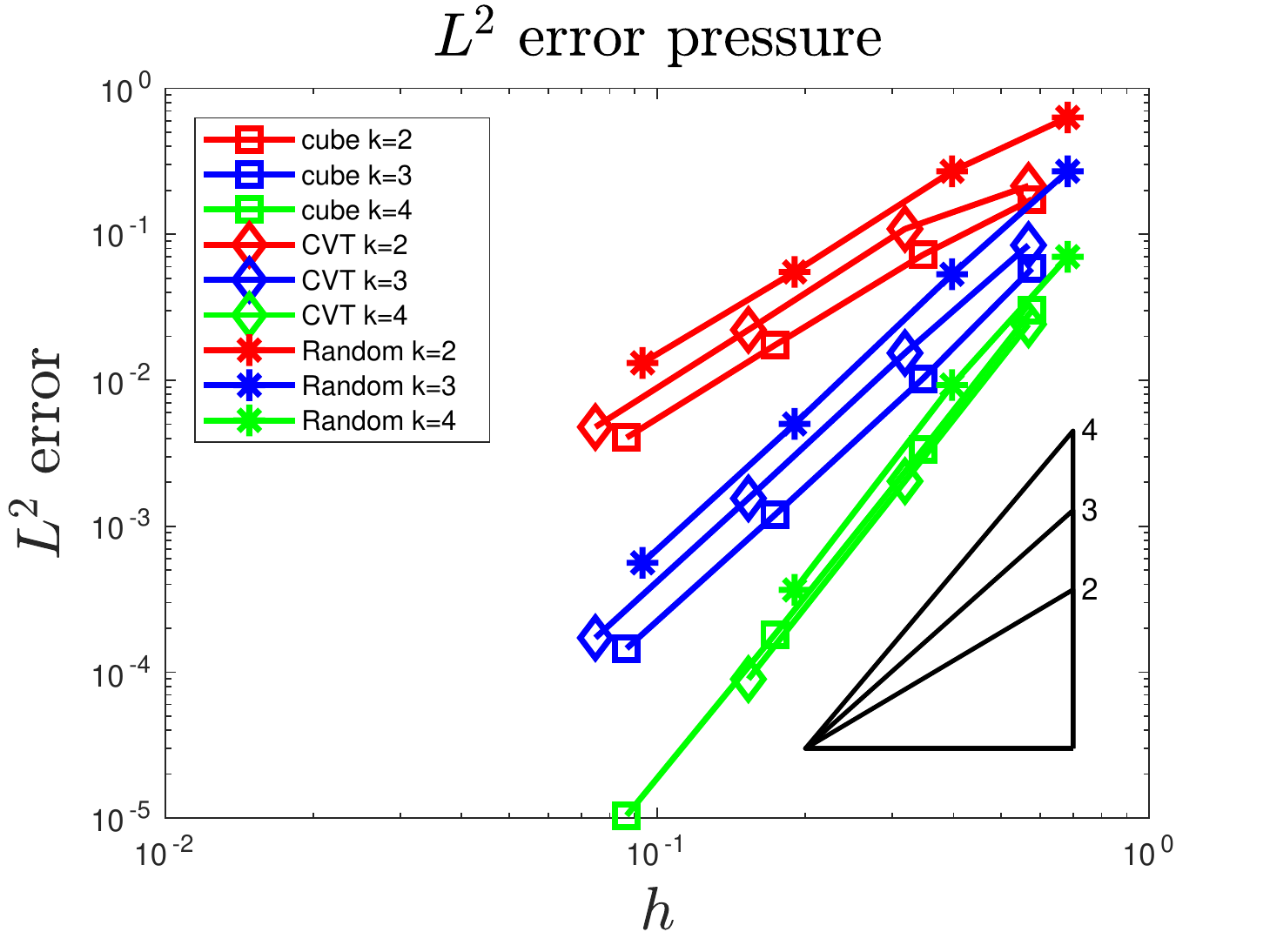}\\
\end{tabular}
\end{center}
\caption{Example 1 Stokes problem: convergence lines for \textbf{Structured}, \textbf{CVT} and \textbf{Random} meshes 
and degrees $k=2,3$ and 4.}
\label{fig:ese1Conv}
\end{figure}

\paragraph{Example 2 (Navier--Stokes problem).}
In this subsection we consider the Navier--Stokes problem described in Equation~\eqref{eq:ns virtual} with Dirichlet boundary conditions.
We consider the same discretization of the unit cube of the previuos example, i.e. the set of meshes \textbf{Structured}, \textbf{CVT} and \textbf{Random}.
We define the right hand side and the boundary conditions in such a way that the exact solution is 
$$
\uu(x,\,y,\,z):=
\left(
\begin{array}{r}
\phantom{-2}\sin(\pi x)\,\cos(\pi y)\,\cos(\pi z)\\
\phantom{-2}\cos(\pi x)\,\sin(\pi y)\,\cos(\pi z)\\
{-2}\cos(\pi x)\,\cos(\pi y)\,\sin(\pi z)
\end{array}
\right)
$$
and
$$
p(x,\,y,\,z):= \sin(2\pi x)\,\sin(2\pi y)\,\sin(2 \pi z)\,.
$$

We solve the nonlinear problem by using standard Newton-Rapson iterations with a stopping criterion based 
on the displacement convergence test error with a tolerance $\texttt{tol=1e-10}$, i.e. until 
$||\X_n-\X_{n+1}||<\texttt{tol}\,||\X_n||$
where $\X_n$ refers to the solution at the $n$-step. 
In Figure~\ref{fig:ese2Conv} we show the convergence lines of the $H^1$ error on the velocity and the $L^2$ error on the pressure, respectively. 
In all these cases we have the predicted trend: $h^{k}$ for the velocity and $h^k$ for the pressure, see Theorem~\ref{thm:u}.
Moreover, also in this case the lines are close to each other varying the mesh discretization, expecially for the velocity solution. Note that, for the pressure solution and random meshes, higher order case $k=3$, there seems to be a loss of accuracy at the second step. We believe this is due to difficulties related to the Newton convergence iterates (the associated linear system  getting quite badly conditioned) since random meshes have a very bad geometry and we are reaching near the memory limit of our platform. We where unable to run a further step due to memory limits. Improving this aspect, possibly by exploring more advanced solvers or changing the adopted virtual element basis \cite{dassi-mascotto:2018}, is beyond the scope of the present paper.  

\begin{figure}[!htb]
\begin{center}
\begin{tabular}{cc}
\includegraphics[height=0.37\textwidth]{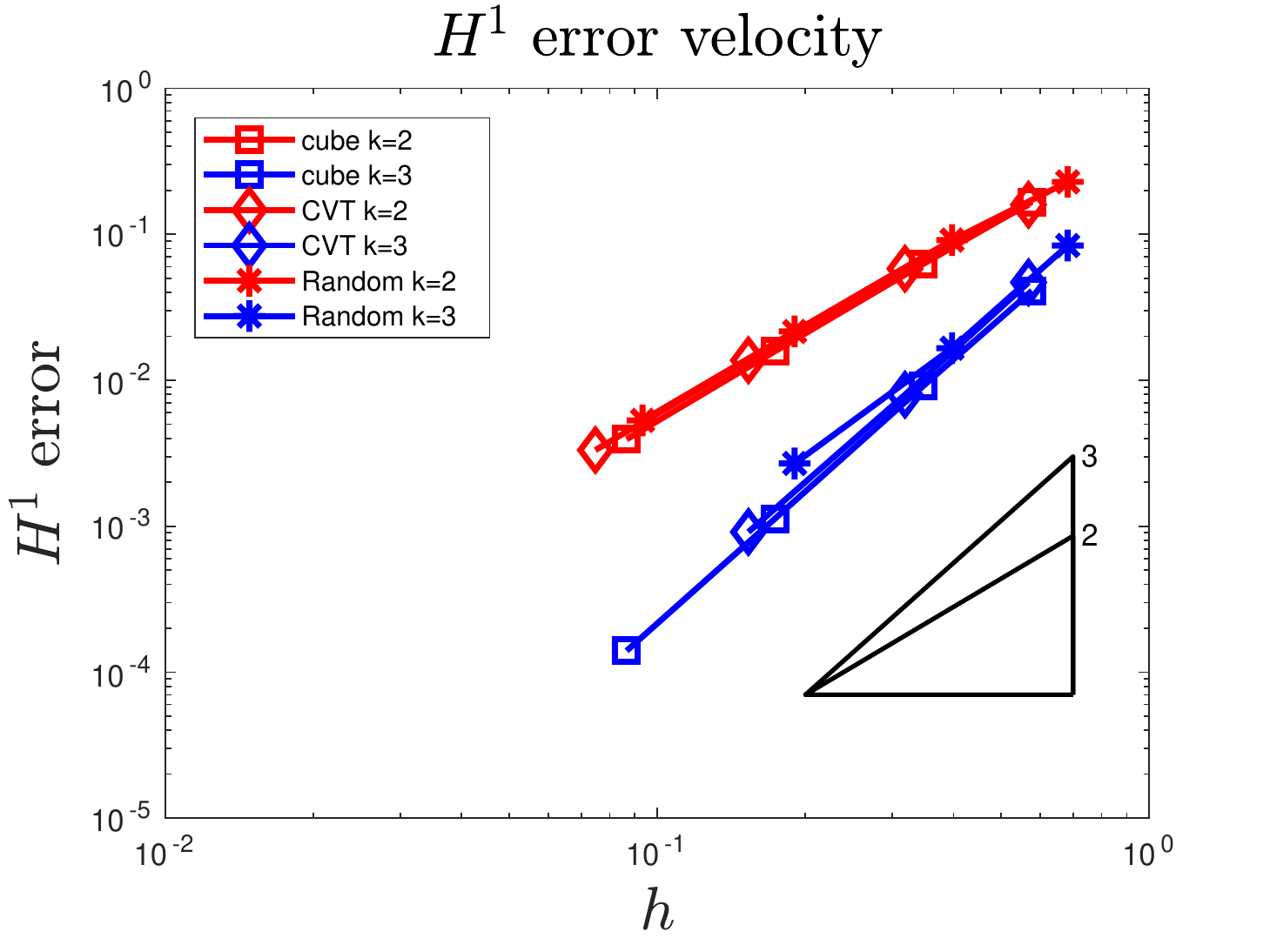}&
\includegraphics[height=0.37\textwidth]{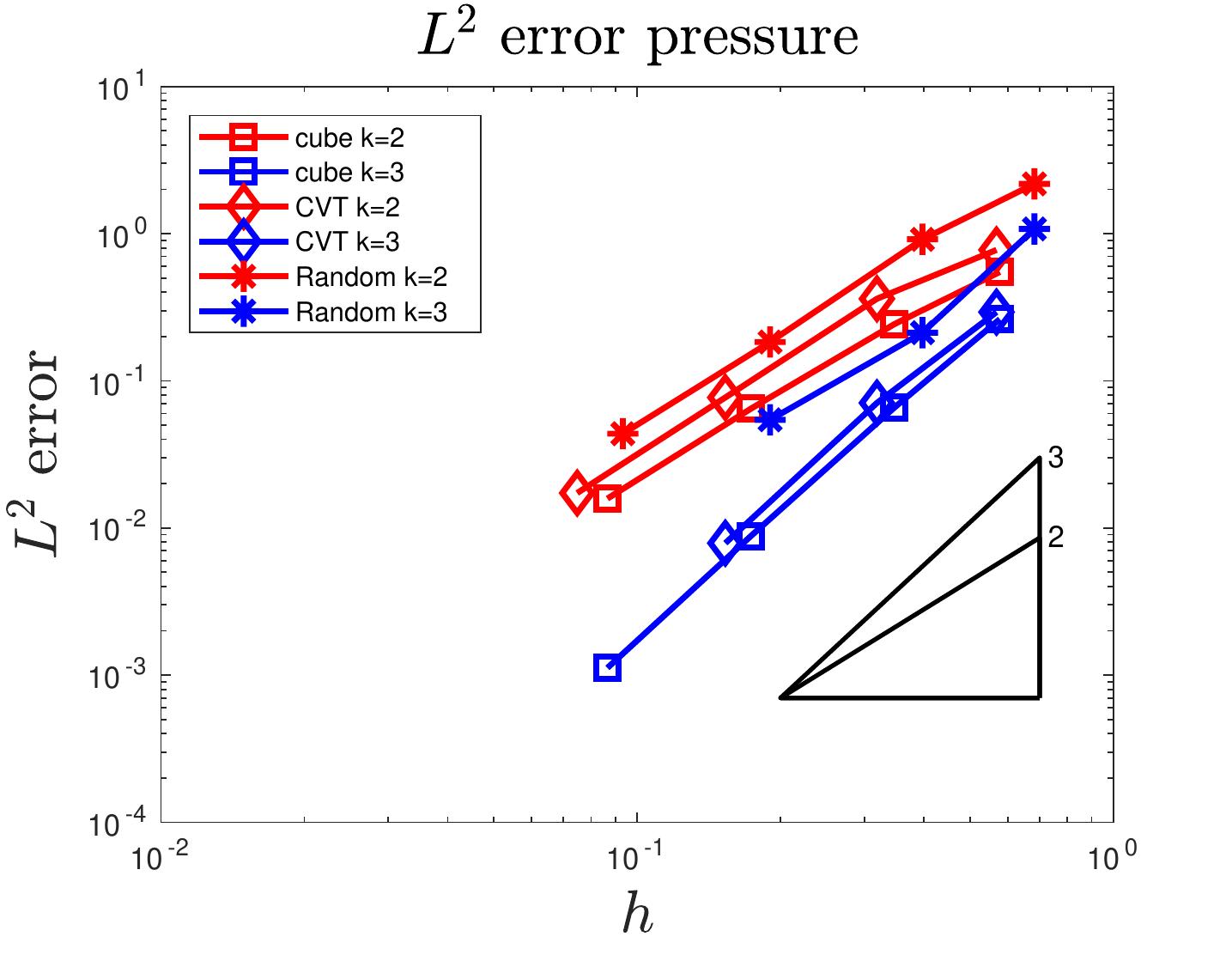}\\
\end{tabular}
\end{center}
\caption{Example 2 Navier--Stokes problem: convergence line for \textbf{Structured}, \textbf{CVT} and \textbf{Random} meshes 
and degrees $k=2$ and 3.}
\label{fig:ese2Conv}
\end{figure}

\paragraph{Example 3 (Benchmark Problems).}
In this paragraph, inspired by \cite{benchmark}, we consider a particular example to numerically show the an advantage of the proposed method.
It is well known that the error on the velocity field of standard inf-sup stable elements for the Stokes equation 
is pressure dependent~\cite{boffi-brezzi-fortin:book}. 
Consequently, the accuracy of the discrete solution $\uu_h$ is affected by the discrete pressure error.
As already shown for the two-dimensional case in~\cite{BLV:2017}, also 
in the three-dimensional case we do not have such dependency on the error, i.e.
the error on the discrete velocity field $\uu_h$ does not depend on the pressure, but 
\emph{only} on the velocity $\uu$ and on the load term $\ff$ (see Theorem~\ref{thm:u}, estimate \eqref{eq:thm:u}).
Note that the present method, although div-free, is not pressure-robust in the sense of~\cite{benchmark} since 
the error on the velocities is indirectly affected by the pressure through the loading approximation term~\cite{BLV:2017}.
Nevertheless it is still much better then the inf-sup stable element in this respect, as the accuracy of the load approximation 
(being a known quantity) can be easily improved.

To numerically verify such property we consider two Stokes problems where 
the exact velocity field is contained in $\textbf{V}_h$ 
\begin{equation*}
\uu(x,\,y,\,z):=
\left(
\begin{array}{c}
k\,x\,z^{k-1}\\
k\,y\,z^{k-1}\\
(2-k)\,x^k+(2-k)\,y^k-2\,z^k
\end{array}
\right)\,,
\end{equation*}
where $k$ is the VEM approximation degree, but we vary the solution on the pressure.
More specifically we will consider these two pressure solutions: 
a polynomial pressure
$$
p_1(x,\,y,\,z) := x^k\,y+y^k\,z+z^k\,x-\frac{3}{2(k+1)}\,,
$$
and an analytic pressure 
$$
p_2(x,\,y,\,z) := \sin(2\pi x)\,\sin(2\pi y)\,\sin(2\pi z)\,.
$$

Note that in both cases, since $p_i\not\in Q_h$ for $i=1,2$, 
a standard inf-sup stable element of analogous polynomial degree would obtain $O(h^k)$ error for the velocities in the $H^1$ norm
even if $\uu\in\textbf{V}_h$.
In the first case, the velocity is a polynomial vector field of degree $k$,
while the pressure is a polynomial of degree $k$ and 
the load term $\ff$ is a polynomial of degree $k$.
In such configuration the presented VEM scheme yields the exact solution up to machine precision for the velocity field.
Indeed, the velocity virtual element space contains polynomials of degree $k$ and,
since the load term is a polynomial of degree $k$, 
the term $\mathcal{H}(\ff,\nu)$ in Equation~\eqref{eq:thm:u} is close to the machine precision, i.e. 
we approximate exactly the load term $\ff$ (cf. Definition \eqref{eq:fhP}),
so the error on $\uu$ is close to the machine precision.

In table of Figure~\ref{fig:ese3Poly} left, we collect the errors $e_{H^1}^{\uu}$ only for the coarsest meshes 
composed by 27 and 68 elements for the \textbf{Structured} and \textbf{Tetra} meshes, respectively.

\begin{figure}[!htb]
\begin{minipage}{0.5\textwidth}
\begin{center}
\begin{tabular}{|c|c|c|}
\multicolumn{3}{c}{$H^1$ error velocity}\\
\hline
$k$ &\textbf{Structured} &\textbf{Tetra} \\
\hline
\texttt{2} &\texttt{1.0576e-13} &\texttt{7.2075e-13} \\
\texttt{3} &\texttt{2.7333e-13} &\texttt{1.1927e-12} \\
\texttt{4} &\texttt{1.5266e-12} &\texttt{2.2718e-10} \\
\hline
\end{tabular}
\end{center}
\end{minipage}
\begin{minipage}{0.5\textwidth}
\begin{center}
\includegraphics[width=1.0\textwidth]{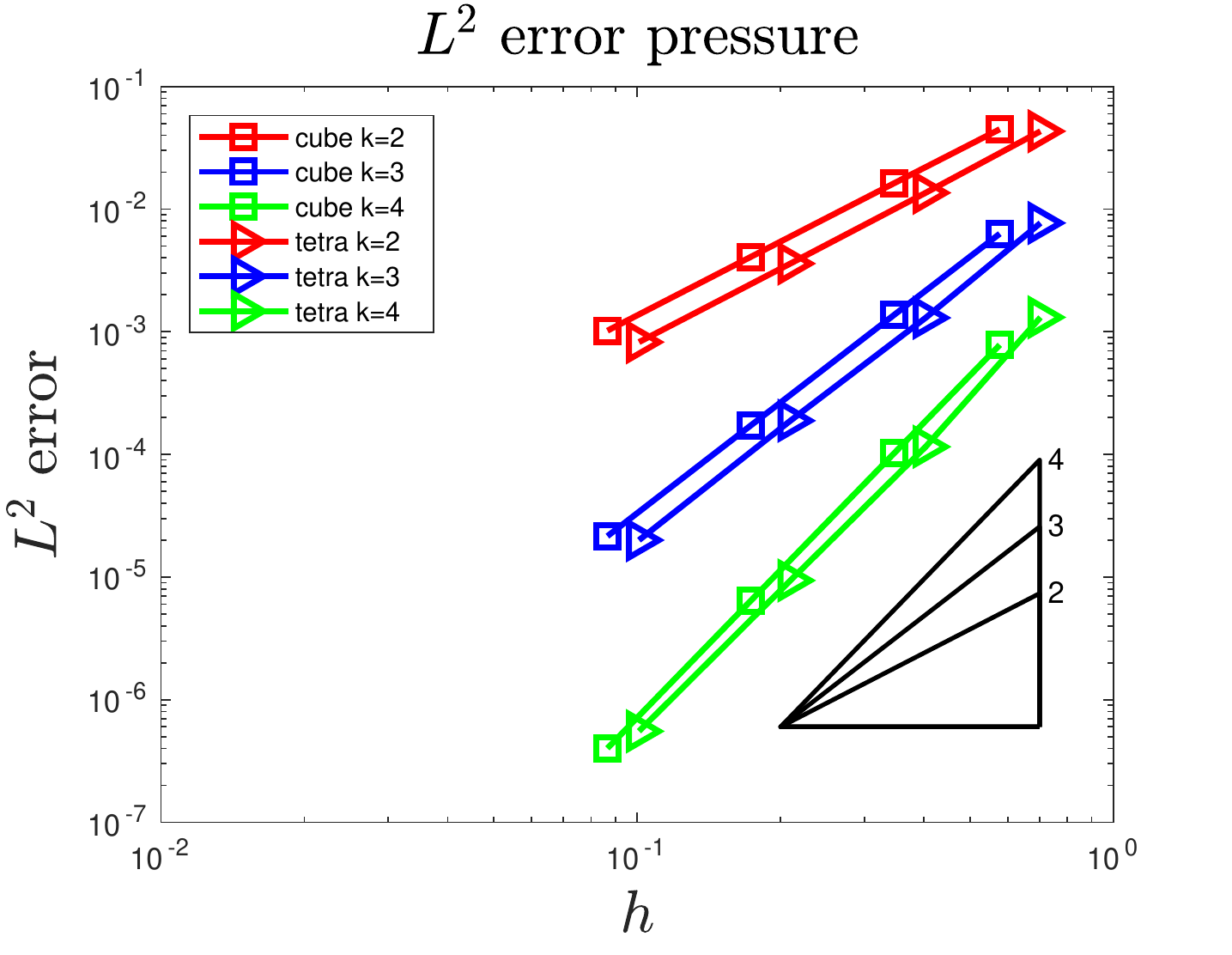}
\end{center}
\end{minipage}

\caption{Example 3 Benchmark problem: the values of the errors $e_{H^1}^{\uu}$ for the coarsest meshes of \textbf{Structured} and 
\textbf{Tetra} meshes, left, pressure convergence lines, right, for a Stokes problem where we consider $p_1$ as pressure.}
\label{fig:ese3Poly}
\end{figure}

In the second case the velocity is still a polynomial of degree $k$, but,
since the pressure is a sinusoidal function, now the right hand side $\ff$ is not a polynomial.
Even if the velocity virtual element space contains the polynomial of degree $k$,
the error is affected by the term $\mathcal{H}(\ff,\nu)$ in Equation~\eqref{eq:thm:u}, i.e.
it is affected by the polynomial approximation we make of the load term so 
the expected error for $e_{H^1}^{\uu}$ is $h^{k+2}$, which is still much better than $O(h^k)$.

In Figure~\ref{fig:ese3Sin} we show the convergence lines for both $e_{H^1}^{\uu}$ and $e_{L^2}^{p}$.
The error trends are the expected ones: we get $O(h^4)$, $O(h^5)$ and $O(h^6)$ for degrees $k=2,3$ and 4, respectively,
while we get $O(h^k)$ for the pressure. In the last step of the $H^1$ norm error, the error is higher than expected  
(this behaviour is  due to machine algebra effects since we are in a range of very small errors).

\begin{figure}[!htb]
\begin{center}
\begin{tabular}{cc}
\includegraphics[height=0.37\textwidth]{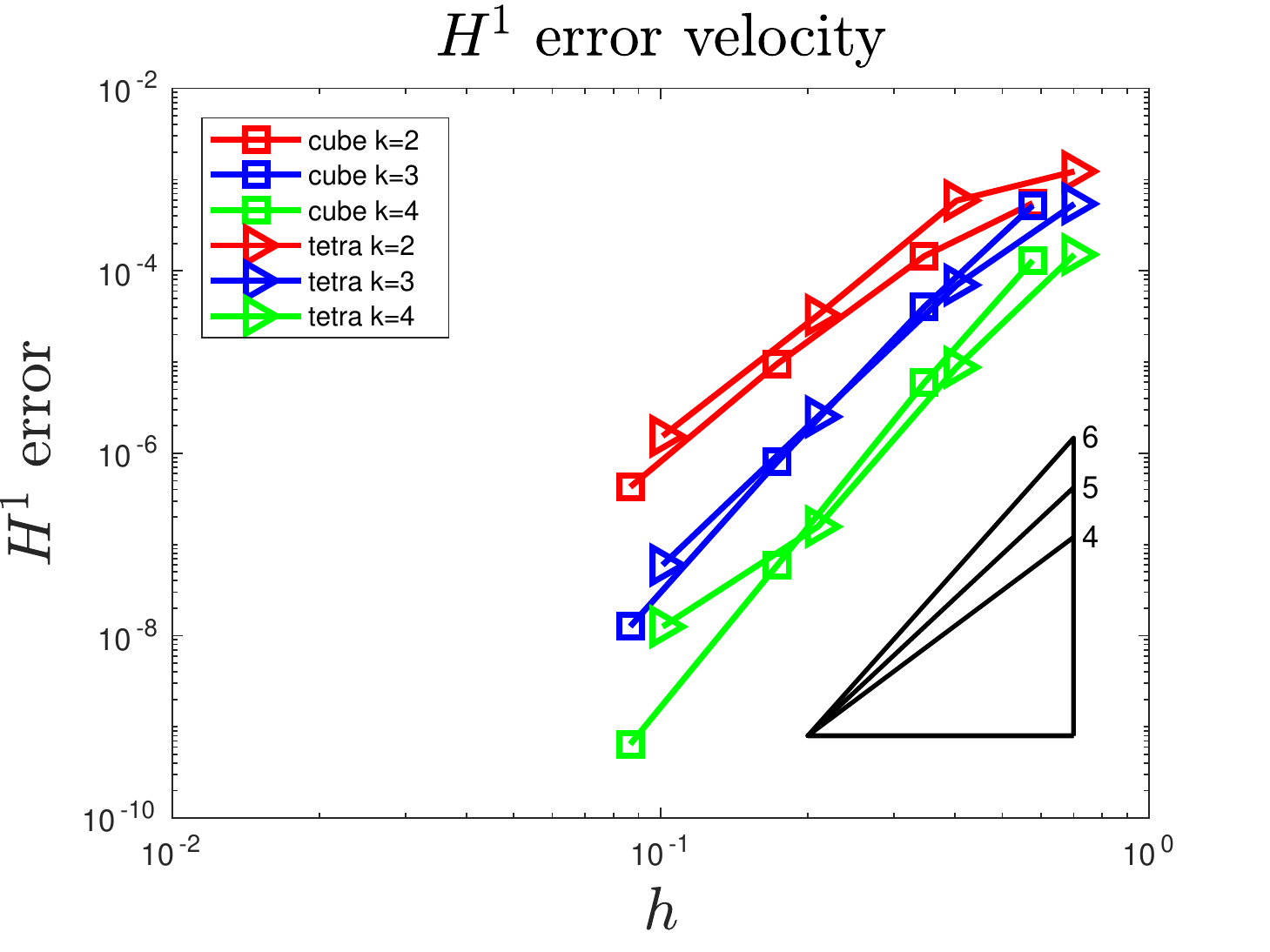}&
\includegraphics[height=0.37\textwidth]{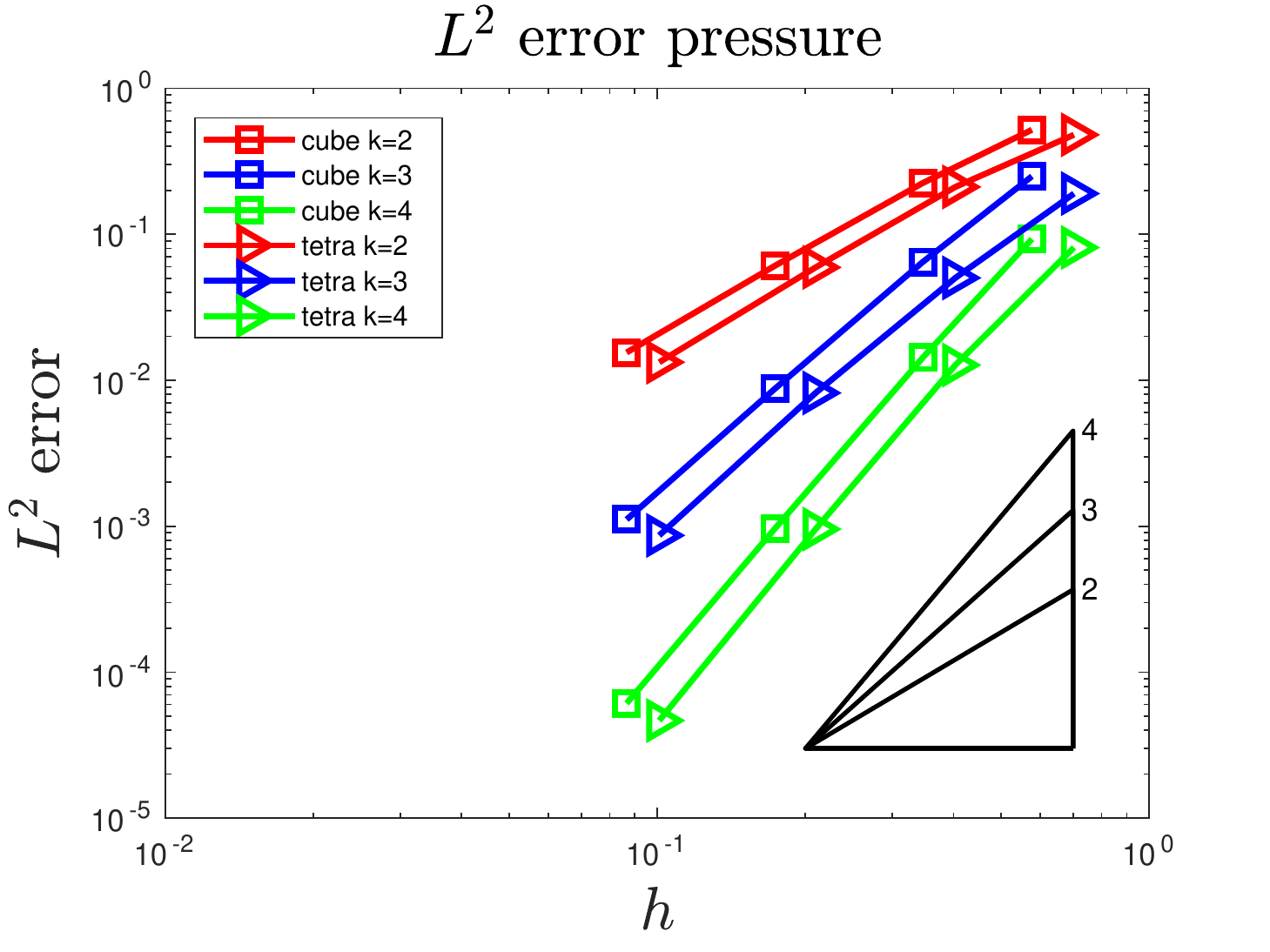}\\
\end{tabular}
\end{center}
\caption{Example 3 Benchmark: convergence lines for a Stokes problem with \textbf{Structured} and
\textbf{Tetra} meshes where we consider a sinusoidal pressure function, $p_2$.}
\label{fig:ese3Sin}
\end{figure}

\section*{Appendix}

The aim of this appendix is addressing the well-posedness of the biharmonic problem with the non homogeneous  boundary conditions stated in Definition \eqref{eq:Sg_h^P}.
Indeed, although in the literature one can find many references for the homogeneous case \cite{bendali-dominguez-gallic:1985, amrouche-et-al:1998, girault-raviart:book}, to the authors best knowledge the extension to the non homogeneous case is labeled as feasible but never explicited. For completeness, we here provide the details.

We first recall that the space ${\boldsymbol \Psi}(P)$ is provided with the norm \cite{girault-raviart:book}:
\[
\|\ppsi\|_{{\boldsymbol \Psi}(P)}^2 :=
\|\ppsi\|_{0, P}^2 + \|\CC \, \ppsi\|_{1, P}^2 + \|\dd \, \ppsi\|_{1, P}^2 \,.
\]
Moreover, if $P$ is a contractible polyhedron the following bounds hold (Lemma 5.2 \cite{girault-raviart:book})
\begin{equation}
\label{eq:norm_equivalence}
\|\ppsi\|_{0, P}^2 + \|\dl \, \ppsi\|_{0, P}^2 
\lesssim
\|\ppsi\|_{{\boldsymbol \Psi}(P)}^2
\lesssim 
\|\dl \, \ppsi\|_{0, P}^2 
\qquad \text{for all $ \ppsi \in {\boldsymbol \Psi}_0(P)$.}
\end{equation}

We start our analysis by recalling the following result concerning the case of homogeneous boundary conditions (see Lemma 5.1 \cite{girault-raviart:book}).
\begin{lemma}
\label{lm:girault-raviart}
Let $P$ be a contractible polyhedron and let ${\boldsymbol F} \colon {\boldsymbol \Psi}_0(P) \to \R$ be a given continuous functional.
The biharmonic problem coupled with homogeneous boundary conditions
\begin{equation*}
\left \{
\begin{aligned}
& \text{find $\ffi  \in \boldsymbol{\Psi}_0(P)$, such that} \\
& \int_P \dl \, \ffi \cdot \dl \, \ppsi \, {\rm d}P
 = {\boldsymbol F}(\ppsi)
 \qquad & \text{for all $ \ppsi \in \boldsymbol{\Psi}_0(P)$,} 
\end{aligned}
\right .
\end{equation*}
has a unique solution $\ffi$.
\end{lemma}

The next theorem extends the well-posedness result of the previous lemma to the case of inhomogeneous boundary conditions. 

\begin{theorem}
\label{thm:wellposed}
Let $P$ be a contractible polyhedron and let
\begin{itemize}
\item ${\boldsymbol h} \in [L^2(\partial P)]^3$ such that  for any $f$, $f_1$, $f_2 \in \partial P$ and for any $e \subseteq f_1 \cap f_2$  
\[
{\boldsymbol h}_{\tau} \in  H(\dd_f, \, f) \cap H(\rr_f, \, f) 
\qquad \text{and} \qquad 
({\boldsymbol h}_{f_1} \cdot \tf_e)_{|e} = ({\boldsymbol h}_{f_2} \cdot \tf_e)_{|e} \,,
\]
\item ${\boldsymbol g} \in [H^{1/2}(\partial P)]^3$ such that
\begin{equation}
\label{eq:NS_bs_compatibility}
{\boldsymbol g} \cdot \nn_P^f = \rr_f \, {\boldsymbol h}_{\tau} \qquad \text{for any $f \in \partial P$,}
\end{equation}
\item ${\boldsymbol f} \in [L^2(P)]^3 \cap \ZZ(P)$.
\end{itemize}
%
%
The biharmonic problem coupled with the non homogeneous boundary conditions
\begin{equation}
\label{eq:NS_bs}
\left \{
\begin{aligned}
& \text{find $\ffi  \in \boldsymbol{\Psi}(P)$, such that} \\
& \int_P \dl \, \ffi \cdot \dl \, \ppsi \, {\rm d}P
 = \int_P {\boldsymbol f} \cdot  \ppsi \, {\rm d}P 
 \qquad & \text{for all $ \ppsi \in \boldsymbol{\Psi}_0(P)$,} \\
& \int_{\partial P} \ffi \cdot \nn_P \,{\rm d}f = 0 \,, \\
& \ffi_{\tau} = {\boldsymbol h}_{\tau}  \qquad & \text{on $\partial P$,} \\
& \CC \, \ffi = {\boldsymbol g}  \qquad & \text{on $\partial P$,}
\end{aligned}
\right .
\end{equation}
has a unique solution $\ffi$.
\end{theorem}

\begin{proof}
Let us consider the following auxiliary problem
\begin{equation}
\label{eq:NS_bs_partial}
\left \{
\begin{aligned}
& \text{find $\ffi^{\partial} \in \boldsymbol{\Psi}(P)$, such that} \\
& 
\begin{aligned}
& \int_{\partial P} \ffi^{\partial} \cdot \nn_P \, {\rm d}f = 0 \,, \\
& \ffi_{\tau}^{\partial} = {\boldsymbol h}_{\tau}  \qquad & \text{on $\partial P$,} \\
& \CC \, \ffi^{\partial} = {\boldsymbol g}  \qquad & \text{on $\partial P$.}
\end{aligned}
\end{aligned}
\right .
\end{equation}
We construct by hand a suitable $\ffi^{\partial}$ that satisfies \eqref{eq:NS_bs_partial}.

Let us consider the Stokes-type problem defined on $P$
\begin{equation}
\label{eq:NS_bs_stokes}
\left\{
\begin{aligned}
& \text{find $(\uu, p) \in [H^1(P)]^3 \times L^2_0(P)$, such that} \\
&
\begin{aligned}
& - \dl \, \uu + \nabla p = {\boldsymbol 0} 
\qquad & \text{in $P$}
\\
&  \dd \, \uu = 0 
\qquad & \text{in $P$}
\\
& \uu = {\boldsymbol g} 
\qquad & \text{on $\partial P$}
\end{aligned}
\end{aligned}
\right.
\end{equation}
then by Theorem 3.4 \cite{girault-raviart:book}, 
there exists a vector potential $\ffi^{\boldsymbol g}$ (possibly not unique) satisfying 
\begin{equation}
\label{eq:NS_bs_ffig}
\left\{
\begin{aligned}
& \text{$\ffi^{\boldsymbol g} \in {\boldsymbol \Psi}(P)$, such that} \\
&
\begin{aligned}
& \CC \, \ffi^{\boldsymbol g}  = \uu 
\qquad & \text{in $P$,}
\\
&  \dd \, \ffi^{\boldsymbol g} = 0 
\qquad & \text{in $P$.}
\end{aligned}
\end{aligned}
\right.
\end{equation}
Moreover \eqref{eq:NS_bs_ffig} implies that $-\dl \, \ffi^{\boldsymbol g} = \CC \, \uu$, thus the following stability estimate holds \cite{boffi-brezzi-fortin:book}
\begin{equation}
\label{eq:NS_bs_ffig_stability}
\|\dl \, \ffi^{\boldsymbol g}\|_{0, P} 
= \|\CC \, \uu\|_{0, P} 
\leq \| \uu\|_{1, P}
\lesssim |{\boldsymbol g}|_{1/2, \partial P} \,.
\end{equation}

Notice that from \eqref{eq:NS_bs_compatibility},  \eqref{eq:compatibility_CC_rr},  \eqref{eq:NS_bs_ffig} and \eqref{eq:NS_bs_stokes}, 
on each face $f \in \partial P$, we infer
\[
\rr_f ({\boldsymbol h}  - \ffi^{\boldsymbol g})_{\tau} = 
\rr_f \, {\boldsymbol h}_{\tau} - \rr_f \, \ffi^{\boldsymbol g}_{\tau} =
{\boldsymbol g} \cdot \nn_P^f - (\CC \, \ffi^{\boldsymbol g})_{|f}\cdot \nn_P^f = 0 \,.
\]
Therefore it can be shown that there exists $\zeta \in H^1(\partial P)$ such that
\begin{equation}
\label{eq:NS_bs_face}
\gr_f \, \zeta_{\tau} = ({\boldsymbol h}  - \ffi^{\boldsymbol g})_{\tau}
\qquad \text{on any $f \in \partial P$.}
\end{equation}
Now we consider the elliptic problem
\begin{equation*}
\left \{
\begin{aligned}
& \Delta \, \omega = 0 \qquad \text{in $P$,} \\
& \omega = \zeta \qquad \text{on $\partial P$.}
\end{aligned}
\right .
\end{equation*}
We observe that $\ffi^{\boldsymbol h} := \gr \, \omega$ satisfies, also recalling \eqref{eq:NS_bs_face},
\begin{equation}
\label{eq:NS_bs_ffih}
\left\{
\begin{aligned}
& \text{$\ffi^{\boldsymbol h} \in {\boldsymbol \Psi}(P)$, such that} \\
&
\begin{aligned}
& \CC \, \ffi^{\boldsymbol h}  = {\boldsymbol 0} 
\qquad & \text{in $P$,}
\\
&  \dd \, \ffi^{\boldsymbol h} = 0 
\qquad & \text{in $P$,}
\\
& \ffi^{\boldsymbol h}_{\tau} = {\boldsymbol h} - \ffi_{\tau}^{\boldsymbol g}
\qquad & \text{on $\partial P$.}
\end{aligned}
\end{aligned}
\right.
\end{equation}
From \eqref{eq:NS_bs_ffih} it holds that
\begin{equation}
\label{eq:NS_bs_ffih_stability}
\dl \, \ffi^{\boldsymbol h} = {\boldsymbol 0} \,.
\end{equation}
By construction $\ffi^{\partial} := \ffi^{\boldsymbol g} + \ffi^{\boldsymbol h}$ satisfies \eqref{eq:NS_bs_partial} and from \eqref{eq:NS_bs_ffig_stability} and \eqref{eq:NS_bs_ffih_stability} it holds that
\begin{equation}
\label{eq:NS_bs_ffip_stability}
\| \dl \, \ffi^{\partial}\|_{0, P} \lesssim |{\boldsymbol g}|_{1/2, \partial P} \,.
\end{equation}
We consider now the  homogeneous auxiliary problem
\begin{equation}
\label{eq:NS_bs_0}
\left \{
\begin{aligned}
& \text{find $\ffi^{\rm hom}  \in \boldsymbol{\Psi}_0(P)$, such that} \\
& \int_P \dl \, \ffi^{\rm hom}\cdot \dl \, \ppsi \, {\rm d}P
 = \int_P {\boldsymbol f} \cdot  \ppsi \, {\rm d}P
  -  \int_P \dl \, \ffi^{\partial} \cdot \dl \, \ppsi \, {\rm d}P
 \qquad & \text{for all $ \ppsi \in \boldsymbol{\Psi}_0(P)$.} 
\end{aligned}
\right .
\end{equation}
Being ${\boldsymbol f} \in [L^2(P)]^3$, from \eqref{eq:NS_bs_ffip_stability}, \eqref{eq:norm_equivalence} and Lemma \ref{lm:girault-raviart}, Problem \eqref{eq:NS_bs_0} has a unique solution $\ffi^{\rm hom}  \in \boldsymbol{\Psi}_0(P)$.
It is straightforward to see that $\ffi := \ffi^{\rm hom} + \ffi^{\partial}$ is a solution to Problem \eqref{eq:NS_bs}.
The uniqueness easily follows from the norm equivalence \eqref{eq:norm_equivalence}.
\end{proof}



\section*{Acknowledgements}
The authors  were
partially supported by the European Research Council through
the H2020 Consolidator Grant (grant no. 681162) CAVE,
Challenges and Advancements in Virtual Elements.
This support is gratefully acknowledged.

\addcontentsline{toc}{section}{\refname}
\bibliographystyle{plain}
\bibliography{biblio}

\end{document}